\documentclass[final,onefignum,onetabnum,letterpaper]{siamonline190516}
\usepackage[scale=0.8]{geometry} 

\usepackage{lipsum}
\usepackage{amsfonts}
\usepackage{epstopdf}
\ifpdf
  \DeclareGraphicsExtensions{.eps,.pdf,.png,.jpg}
\else
  \DeclareGraphicsExtensions{.eps}
\fi

\usepackage{datetime}
\newdateformat{monthyeardate}{%
  \monthname[\THEMONTH] \THEDAY, \THEYEAR}

\usepackage{academicons}
\usepackage{xcolor}
\renewcommand{\orcid}[1]{\href{https://orcid.org/#1}{\textcolor[HTML]{A6CE39}{orcid.org/#1}}}

\usepackage{amsmath}
\allowdisplaybreaks
\usepackage{amssymb}
\usepackage{commath}
\usepackage{mathtools}
\usepackage{bbm}

\usepackage{color}
\usepackage{graphicx}
\graphicspath{{figures/}}
\usepackage[small]{caption}
\usepackage{subcaption}

\usepackage{relsize}
\usepackage{adjustbox}
\usepackage{algorithm}
\usepackage[noend]{algpseudocode}
\usepackage{booktabs}
\usepackage{tikz}

\usepackage{verbatim}

\usepackage{cleveref}

\usepackage{pdflscape}
\usepackage{afterpage}
\usepackage{capt-of}

\usepackage{enumitem}
\setlist[enumerate]{leftmargin=.5in}
\setlist[itemize]{leftmargin=.5in}

\newsiamthm{problem}{Problem}
\newsiamremark{remark}{Remark}
\newsiamremark{hypothesis}{Hypothesis}
\crefname{hypothesis}{Hypothesis}{Hypotheses}
\newsiamremark{example}{Example}
\newsiamthm{claim}{Claim}
\newsiamthm{conjecture}{Conjecture}

\headers{Beyond the SBP property: nullspace consistency, sparsity, and regularization}{Glaubitz, Iske, Lampert, and \"Offner}

\title{Beyond the summation-by-parts property: nullspace consistency, sparsity, and regularization for FSBP operators
\thanks{
\monthyeardate\today
\corresponding{Joshua Lampert}
}}

\author{
Jan Glaubitz\thanks{Department of Mathematics, Linköping University, Sweden
(\email{jan.glaubitz@liu.se}, \orcid{0000-0002-3434-5563})
}
\and
Armin Iske\thanks{Department of Mathematics, University of Hamburg, Germany
(\email{armin.iske@uni-hamburg.de}, \orcid{0000-0003-1743-5484} and \email{joshua.lampert@uni-hamburg.de}, \orcid{0009-0007-0971-6709})
}
\and
Joshua Lampert\footnotemark[3]
\and
Philipp \"Offner\thanks{
Institute of Mathematics, Clausthal University of Technology, Germany
(\email{mail@philippoeffner.de}, \orcid{0000-0002-1367-1917})
}
}

\usepackage{amsopn}

\DeclareMathOperator{\rank}{rank}
\DeclareMathOperator{\cond}{cond}
\DeclareMathOperator{\diag}{diag}

\DeclareMathOperator{\spn}{span} 

\DeclareMathOperator{\nul}{null} 
\DeclareMathOperator{\spec}{spec}
\newcommand{\scp}[2]{\left\langle{#1, #2}\right\rangle}

\renewcommand{\dim}{\mathrm{dim} \,}

\newcommand{\Z}{\mathbb{Z}}
\newcommand{\R}{\mathbb{R}}
\newcommand{\C}{\mathbb{C}}
\renewcommand{\O}{\mathcal{O}}
\newcommand{\F}{\mathcal{F}}
\newcommand{\G}{\mathcal{G}}

\usepackage{lineno}

\ifpdf
\hypersetup{
  pdftitle={Beyond the SBP property},
  pdfauthor={Jan Glaubitz}
}
\fi

\begin{document}

\maketitle

\begin{abstract}
	We investigate the construction and performance of summation-by-parts (SBP) operators, which offer a powerful framework for the systematic
development of structure-preserving numerical discretizations of partial differential equations. Previous approaches for the construction of
SBP operators have usually relied on either local methods or sparse differentiation matrices, as commonly used in finite difference schemes.
However, these methods often impose implicit requirements that are not part of the formal SBP definition. We demonstrate that adherence to
the SBP definition alone does not guarantee the desired accuracy, and we make additional conditions explicit that SBP operators need
to satisfy in order to achieve accuracy. While these conditions are known in the SBP literature, they are usually enforced only implicitly
by the respective construction procedure. Specifically, we analyze the error minimization for an augmented basis, discuss the role of
sparsity, and examine the importance of nullspace consistency in the construction of SBP operators. A dispersion and dissipation
analysis shows that the loss of accuracy has two sources. First, a lack of nullspace consistency produces stationary modes, which are
neither transported nor damped by the scheme and whose contribution to the error converges more slowly under mesh refinement. Second, the
physical mode, i.e., the discrete approximation of an exact traveling wave, can be poorly resolved, and the resulting error dominates in
long-time simulations. Furthermore, we show how these design criteria can be integrated into a recently proposed optimization-based
construction procedure for function space SBP (FSBP) operators on arbitrary grids. Our findings are supported by numerical experiments that
illustrate the improved accuracy of the numerical solutions obtained with the proposed SBP operators.

\end{abstract}



\begin{Code}
    \url{https://github.com/JoshuaLampert/2026_SBP_not_enough}
\end{Code}

\begin{DOI}
	Not yet assigned
\end{DOI}

\section{Introduction}
\label{sec:introduction}

Summation-by-parts (SBP) operators have emerged as a powerful tool in the numerical approximation of partial differential equations (PDEs),
particularly conservation laws of the form
\begin{equation}\label{eq:CL}
	\partial_t \mathbf{u}(t,\mathbf{x}) + \nabla \cdot \mathbf{f}(\mathbf{u}) = 0,
	\quad t>0, \quad \mathbf{x} \in \Omega,
\end{equation}
with conserved variable $\mathbf{u}: \R_{>0} \times \Omega \to \R^m$, flux function $\mathbf{f}: \R^m \to \R^{m \times d}$, and $\Omega
\subset \R^d$. Conservation laws are fundamental in modeling a wide range of physical phenomena, including fluid dynamics, electromagnetism,
and traffic flow. Accurate and stable numerical methods for solving conservation laws are thus crucial for reliable simulations in science
and engineering. SBP operators mimic the integration-by-parts property at the discrete level, which is essential for ensuring conservation
and stability in numerical schemes. Many classes of discretizations for hyperbolic conservation laws can be formulated in terms of SBP
operators, including finite difference (FD) \cite{kreiss1974finite,strand1994summation,mattsson2004summation}, finite volume
\cite{nordstrom2001finite,nordstrom2003finite}, continuous Galerkin
\cite{hicken2016multidimensional,hicken2020entropy,abgrall2020analysisI}, discontinuous Galerkin (DG)
\cite{gassner2013skew,carpenter2014entropy,chan2018discretely}, radial basis function (RBF) \cite{glaubitz2024energy}, flux reconstruction
\cite{ranocha2016summation,offner2018stability}, active flux \cite{barsukow2026stability,abgrall2026new}, and overset-grid
\cite{glaubitz2027towards} methods. See also the review articles \cite{delreyfernandez2014review,svard2014review} for more details.

Traditionally, SBP operators have been constructed based on exactness with respect to polynomial bases. Recently, this has been extended to
general finite-dimensional function spaces \cite{glaubitz2023summation,glaubitz2023multi,glaubitz2025optimization} yielding function space
SBP (FSBP) operators, which we will build upon to keep the discussion general. In the design of SBP operators, additional criteria that go
beyond the SBP definition usually need to be considered. However, these criteria are not always explicitly stated and generally vary between
different applications and discretization methods.

Many of them are well established. Nullspace consistency, i.e., the requirement that the nullspace of the differentiation matrix contains
the constants only, is a standard assumption in convergence proofs for energy-stable schemes \cite{svard2019convergence}. It is equivalent
to the invertibility of a rank-one modification of the differentiation matrix, and much of the convergence theory for SBP methods relies on
the stronger eigenvalue property of this matrix \cite{linders2022eigenvalue}. Furthermore, SBP operators that are not nullspace consistent
have been reported before \cite{ranocha2019notes,linders2020properties,linders2022eigenvalue}. Likewise, locality and sparsity of the
interior stencils together with suitable boundary closures are constitutive for classical FD-SBP operators
\cite{mattsson2004summation,mattsson2014optimal,mattsson2018boundary}, and the exactness degree of the differentiation matrix as well as the
accuracy of the norm matrix as a quadrature rule are standard design criteria in the generalized SBP framework
\cite{delreyfernandez2014review,hicken2016multidimensional}. What these constructions have in common is that the corresponding properties
hold \emph{implicitly} rather than being imposed as explicit conditions. They arise as a by-product of collocation on a quadrature node set,
or of a repeating interior stencil combined with a boundary closure, where the repeating stencil relies on the translation invariance of
polynomial spaces. Consequently, they are rarely stated as requirements in their own right. Therefore, it is desirable to have an overview
of these additional conditions and a general framework for the design of SBP operators that can be tailored to specific applications.

\subsection*{Our contribution}

We make the additional conditions above explicit and show how they can be built into the construction of FSBP operators. Our starting point
is the recently proposed optimization-based construction procedure \cite{glaubitz2025optimization}, which is formulated directly in terms of
\Cref{def:FSBP} and, for this reason, applies to arbitrary node sets and function spaces. This generality is exactly what makes the
procedure attractive. At the same time, the conditions discussed above are independent of \Cref{def:FSBP} and are therefore not inherited
automatically, as they are in the classical constructions. Supplying them is the purpose of this work. Specifically, we
\begin{enumerate}
	\item[(i)]
	demonstrate by explicit examples that FSBP operators can be highly inaccurate for hyperbolic conservation laws even when they are exact
	for polynomials of a high degree, and that such operators are obtained routinely rather than exceptionally if no further structure is
	imposed;

	\item[(ii)]
	trace this behavior back, by a dispersion and dissipation analysis, to two sources, stationary modes and a poorly resolved physical
	mode. This is the first dispersion and dissipation study for nonpolynomial FSBP operators. Stationary
	modes are nullmodes of the differentiation matrix that vanish at the boundary nodes and are therefore neither transported nor damped.
	They arise from a lack of nullspace consistency, and their contribution to the error converges more slowly under mesh refinement. The
	physical mode is the discrete approximation of an exact traveling wave, and its error dominates in long-time simulations. We also
	explain the lack of nullspace consistency by a rank bound for the minimum-norm solution of the construction problem;

	\item[(iii)]
	extend the optimization-based construction by enforcing a sparsity pattern on the differentiation matrix and by regularizing the
	optimization problem with an augmented basis, and demonstrate that both considerably improve the accuracy; and

	\item[(iv)]
	show that the extended framework reproduces classical polynomial SBP operators for suitable choices of nodes, function space, and
	sparsity pattern, so that it generalizes rather than replaces the established constructions.
\end{enumerate}

We restrict the discussion to SBP operators in one spatial dimension and use tensor products for higher dimensions.

\subsection*{Outline}
\Cref{sec:background} reviews the theory of SBP operators and existing construction methods. Illustrative examples of SBP operators meeting
the standard conditions but failing in practical applications are provided in \Cref{sec:counterexample}, which demonstrates the necessity of
additional criteria. \Cref{sec:optimization} discusses different strategies to enforce additional structure on SBP operators and how to
incorporate these into a recently proposed optimization-based construction method. Before doing so, it explains why the unconstrained
construction can return operators that are not nullspace consistent. It also presents numerical experiments demonstrating the effectiveness
of the proposed approaches. It closes with a dispersion and dissipation analysis of the resulting schemes and a study of their
sensitivity to the initialization of the optimization. Finally, \Cref{sec:summary} summarizes the findings and outlines directions for
future research.

\section{Background}
\label{sec:background}

For a set of one-dimensional nodes $X = \{x_i\}_{i = 1}^N$ and a function $u$, we write $\mathbf{x} = [x_1, \ldots, x_N]^T$, $\mathbf{u} =
[u(x_1), \ldots, u(x_N)]^T$, and $\mathbf{u}' = [u'(x_1), \ldots, u'(x_N)]^T$ for its derivative. Moreover, $\mathbf{e}_L = [1, 0, \ldots,
0]^T$ and $\mathbf{e}_R = [0, \ldots, 0, 1]^T$ denote the left and right boundary basis vectors.

We base our discussion on the recently proposed extension of SBP operators to general finite-dimensional function spaces
\cite{glaubitz2023summation,glaubitz2024energy,glaubitz2025optimization}.
\begin{definition}[FSBP operators]\label{def:FSBP}
	Let $\F$ be a finite-dimensional subspace of $C^1([x_L, x_R])$. We say $D = P^{-1} Q$ is an \emph{$\F$-exact FSBP operator},
	approximating the first derivative operator $\partial_x$ on the node set $X$, if
	\begin{enumerate}
		\item[(i)]
		$D \mathbf{f} = \mathbf{f}'$ for all $f \in \F$,

		\item[(ii)]
		the norm matrix $P$ is symmetric positive definite, and

		\item[(iii)]
		$Q + Q^T = B = \mathbf{e}_R\mathbf{e}_R^T - \mathbf{e}_L\mathbf{e}_L^T$.

	\end{enumerate}
\end{definition}

The first condition ensures that the derivative is exact on the function space $\F$. For traditional SBP operators, $\F$ is usually chosen
as the space of polynomials up to a certain degree. The second condition ensures that $P$ defines a discrete inner product
$\scp{\mathbf{u}}{\mathbf{v}}_P = \mathbf{u}^T P \mathbf{v}$ approximating $\int_{x_L}^{x_R} uv \, dx$. The third condition is the SBP
property, which mimics integration by parts discretely.

SBP operators can be grouped into two classes: (i) dense local operators, which are usually based on collocation and used in finite element
(FE) methods, and (ii) sparse global operators, which are usually employed in FD methods. This distinction is essential in the design of SBP
operators, since the requirements for the two classes differ significantly. Due to the collocation approach, the construction of polynomial
SBP operators from the first class is usually straightforward when restricted to nodes that correspond to a Gaussian quadrature. Examples
include Gauss--Lobatto--Legendre operators \cite{gassner2013skew} or generalized SBP operators based on, e.g., Gauss--Legendre or
Gauss--Radau nodes, which do not include one or both boundary points. See \cite{kopriva2009implementing} for more details.

Classical FD-SBP operators have been constructed for various orders of accuracy and grid configurations, with optimization strategies for
the norm and derivative matrices described in \cite{mattsson2004summation,mattsson2014optimal,mattsson2018boundary}. These procedures build
upon classical FD approximations and result in derivative matrices $D$ with banded structures and boundary closures. However, these
classical FD-SBP operators are limited to polynomial bases and specific grids, which are usually equidistant in the interior.

Recently, an optimization-based procedure for constructing FSBP operators on general node sets and function spaces has been proposed in
\cite{glaubitz2025optimization}. That work focused on dense local operators typically applied in FE methods. It formulates the construction
of FSBP operators as an optimization problem using only the definition of an SBP operator. Since the construction algorithms presented in
this paper build on this approach, we briefly summarize its main idea here, which also fixes the notation.

From now on, we focus on diagonal-norm FSBP operators, i.e., $P$ is a diagonal matrix. Let $S = (Q - Q^T)/2$ be the skew-symmetric part of
$Q$ and let $V$ and $V_x$ be the Vandermonde matrices of a basis $\{f_j\}_{j = 1}^K$ of $\F$ and the derivatives $\{f_j'\}_{j = 1}^K$
evaluated at $X$, i.e., $V_{ij} = f_j(x_i)$ and $(V_x)_{ij} = f_j'(x_i)$. The key idea is to rewrite condition (i) from \Cref{def:FSBP} as
$XW + BV/2 = 0$, where $X = [S, P]\in\R^{N \times 2N}$ is a matrix of the unknowns and $W = [V; -V_x]\in\R^{2N \times K}$ is known. In this
setting, $D = P^{-1} (S + B/2)$ being an $\F$-exact FSBP operator is equivalent to $X$ being a global minimizer of the constrained quadratic
optimization problem
\begin{equation}\label{eq:optimization_problem}
	\min_{X \in \mathcal{X}}\|XW + BV/2\|^2
\end{equation}
with constraints
\begin{equation}\label{eq:constraint_set}
	\mathcal{X} = \{X = [S, P] \mid S^T = -S, P = \diag(p_1, \ldots, p_N), p_i > 0\}
\end{equation}
and $\|XW + BV/2\| = 0$. The constraints can be removed for diagonal-norm operators by parametrizing $P$ and $S$ by vectors
$\boldsymbol{\rho} \in \R^N$ and $\boldsymbol{\sigma} \in \R^L$ as
\begin{equation}\label{eq:diag-sig}
	P(\boldsymbol{\rho}) = \frac{x_R - x_L}{\sum_{j = 1}^N s(\rho_j)}\diag(s(\rho_1), \ldots, s(\rho_N)),
\end{equation}
for the sigmoid logistic function $s(x) = \frac{1}{1 + e^{-x}}$, and
\begin{equation}\label{eq:skew-symmetric}
	S(\boldsymbol{\sigma}) = \begin{pmatrix}
		0 & \sigma_1 & \sigma_2 & \sigma_3 & \dots & \sigma_{N - 1} \\
		-\sigma_1 & 0 & \sigma_N & \sigma_{N + 1} & \dots & \sigma_{2N - 3} \\
		-\sigma_2 & -\sigma_N & 0 & \sigma_{2N - 2} & \dots & \sigma_{3N - 6} \\
		\vdots & \vdots & \vdots & \ddots & \vdots & \vdots \\
		-\sigma_{N - 2} & -\sigma_{2N - 4} & -\sigma_{3N - 7} & -\sigma_{4N - 11} & \ddots & \sigma_{L} \\
		-\sigma_{N - 1} & -\sigma_{2N - 3} & -\sigma_{3N - 6} & -\sigma_{4N - 10} & \dots & 0
	\end{pmatrix}\in\R^{N\times N},
\end{equation}
where $L = N(N - 1)/2$. The sigmoid function ensures that the positivity constraint $p_i > 0$ is automatically satisfied. The
normalization in \eqref{eq:diag-sig} fixes $\sum_i p_i = x_R - x_L$, which does not restrict the construction if $1, x \in \F$, as in all
our examples, since every $\F$-exact FSBP operator then satisfies it. Indeed, the weights $p_i$ of a diagonal norm matrix define a
quadrature formula that is exact for the derivatives of all products $fg$ with $f, g \in \F$ \cite[Theorem~4.4]{glaubitz2023summation}. For
$f = 1$ and $g = x$, this yields exactness for constants, i.e., $\sum_i p_i = x_R - x_L$. Substituting \eqref{eq:diag-sig} and
\eqref{eq:skew-symmetric} into \eqref{eq:optimization_problem} yields the unconstrained optimization problem
\begin{equation}\label{eq:optimization-problem-parametrized}
	\min_{\boldsymbol{\rho} \in \R^N, \boldsymbol{\sigma} \in \R^L} \|X(\boldsymbol{\rho}, \boldsymbol{\sigma}) W + BV/2\|^2.
\end{equation}
Since an $\F$-exact FSBP operator is characterized by $\|XW + BV/2\| = 0$ and every norm vanishes precisely on the zero matrix, the
set of global minimizers of \eqref{eq:optimization_problem} does not depend on the choice of the norm. If no $\F$-exact FSBP operator exists
on the given nodes, however, the minimum is positive and the minimizer does depend on the norm, see \Cref{sub:optimization_regularization}.
In our computational examples, we use the Frobenius norm, which is smooth and inexpensive to evaluate and is therefore well suited to the
gradient-based optimizer described in the following remark.

\begin{remark}
	To solve the optimization problem, we employ the Julia \cite{bezanson2017julia} package Optim.jl \cite{morgensen2018optim}. We use the
	Broyden--Fletcher--Goldfarb--Shanno (BFGS) algorithm, which needs the gradient of the optimization function, but, in contrast to
	the Newton method, uses an approximation of the Hessian matrix. We use automatic differentiation (AD) to compute the gradient. Note that
	the unconstrained optimization problem is not convex, which means that---in theory---the minimum might not be unique and the BFGS
	algorithm might not converge to a global minimum. However, in our experiments this was usually no problem and using global solvers for
	non-convex problems or solving the (convex) constrained problem did not yield any better results and was usually less efficient. A more
	detailed analysis is of interest, but is left for future work.
\end{remark}

\begin{remark}[Numerical conventions]\label{rem:conventions}
	Several statements below rely on numerical thresholds, which we fix here. We refer to the Frobenius norm $\|XW + BV/2\|$, i.e.,
	the square root of the objective in \eqref{eq:optimization_problem}, as the residual. We regard it as zero, and the resulting operator
	as an FSBP operator, if it is smaller than $10^{-9}$; the residuals we obtain are usually far smaller, see \Cref{table:diagnostics}.
	Numerical ranks are computed from the singular values $\sigma_1 \ge \ldots \ge \sigma_N$ of $D$ as the number of $\sigma_i$ exceeding
	$\tau\sigma_1$, with $\tau = N\varepsilon$ and the machine epsilon $\varepsilon \approx 2.2\cdot 10^{-16}$; this is the default of the
	function \texttt{rank} in Julia. The rank is a meaningful quantity here because the singular values split into two clearly separated
	groups. For the dense operator with $\F = \mathcal{P}_9$ on $N = 50$ nodes, for example, $\sigma_{20} = 3.4\cdot 10^{-3}$ and
	$\sigma_{21} = 2.4\cdot 10^{-13}$, and the resulting rank is the same for every $\tau$ between $10^{-6}$ and $10^{-14}$. Since the rank
	is a discontinuous quantity, we additionally report the ratio $\sigma_{N - 1}/\sigma_1$, which indicates how close an operator is to
	being singular beyond its nullspace, and the condition numbers of $P$ and of $\tilde D$ from \eqref{eq:D-tilde}.
\end{remark}

Note that the above construction procedure is general and can be used to construct FSBP operators for arbitrary node sets and function
spaces. However, generally, no additional properties are enforced on the resulting operators, although such properties are often necessary
in practical applications. Examples where this is problematic are presented in the next section.

\section{Counterexamples of FSBP operators}
\label{sec:counterexample}

To motivate the need for additional conditions on SBP operators beyond those specified in \Cref{def:FSBP}, we present some illustrative
examples of SBP operators that satisfy the standard SBP conditions but exhibit undesirable properties.

\begin{example}
	Consider the extreme case, where the differentiation matrix $D$ has rank one and the norm matrix $P$ is nearly singular:
	\begin{equation}
		D = \begin{pmatrix}
			-1 & 0 & \ldots & 0 & 1\\
			-1 & 0 & \ldots & 0 & 1\\
			\vdots & \vdots & \vdots & \vdots & \vdots\\
			-1 & 0 & \ldots & 0 & 1\\
			-1 & 0 & \ldots & 0 & 1
		\end{pmatrix},\quad
		P = \frac{1}{2}\diag(1, \varepsilon, \ldots, \varepsilon, 1).
	\end{equation}
	The above operator defined by $D$ and $P$ on $X = \{(i - 1)/(N - 1) \}_{i = 1, \ldots, N}$ can get arbitrarily close to an SBP operator
	for $\mathcal{F} = \spn\{1, x\}$ when $\varepsilon\to 0$. More precisely, $D$ is exact on $\F$ and $P$ is symmetric positive
	definite for every $\varepsilon > 0$, whereas the SBP property $Q + Q^T = B$ with $Q = PD$ is violated by terms of order $\varepsilon$.
	The residual of \eqref{eq:optimization_problem} for $X = [(Q - Q^T)/2, P]$ is of order $\varepsilon$ as well, because it is weighted by
	the nearly singular matrix $P$. Any construction procedure that accepts an operator once this residual falls below a fixed tolerance
	therefore accepts the above operator. Although this operator numerically is a valid solution to the optimization problem
	\eqref{eq:optimization_problem}--\eqref{eq:constraint_set}, it is evident that the differentiation matrix is not a suitable discrete
	derivative in a non-periodic domain as it only uses information from the boundary nodes. We stress that this example is
	intentionally extreme. It is not representative of the operators that the construction procedure of \Cref{sec:background} returns in
	practice, and its sole purpose is to show that \Cref{def:FSBP} by itself does not guarantee a meaningful approximation of the
	derivative. Operators that are actually produced by the construction procedure, and that fail for less obvious reasons, are considered
	in the remainder of this section.
\end{example}

In the following, we present numerical demonstrations of SBP operators that are exact for higher-order polynomials but still fail to solve
the linear advection equation accurately
\begin{equation}
	\partial_t u + a\partial_x u = 0,
	\quad t \in [0, 1.75], \quad x \in [-1, 1],
\end{equation}
even for smooth solutions. To this end, we construct global FSBP operators using the optimization-based approach described in
\Cref{sec:background} for polynomial bases $\F = \spn\{x^i, i = 0, \ldots, d\}$ of different degrees $d$ on $N=50$ equidistant nodes. As the
initial condition, we use $u_0(x) = \sin(k\pi x)$, set the advection velocity to $a=2$, and use different values for $k$. The spatial
discretization is combined with a third-order strong-stability-preserving Runge--Kutta method with five stages \cite{gottlieb2003strong}. We
use a CFL number of $0.05$ for time integration, which is small enough for the error of the time integration to be negligible
compared with the spatial error. The results are shown in \Cref{fig:advection_solutions} for $k = 1$ (left) and $k = 2$ (right). It can be
observed that the numerical solution using the FSBP operators with polynomial degrees $d \leq 3$ is highly inaccurate, although these SBP
operators are exact for polynomials of degree up to $d$. The effect becomes even more pronounced for higher frequencies in the initial
condition, where also high-order FSBP operators yield inaccurate results. This suggests that unresolved modes are not sufficiently damped.
Note that all numerical solutions are conservative up to machine precision and energy stable due to the SBP property and the inclusion of
constants in the function space \cite{glaubitz2023summation}. This implies that energy stability alone is not sufficient to ensure accurate
results.

\begin{figure}[tb]
	\centering
	\includegraphics[width=\textwidth]{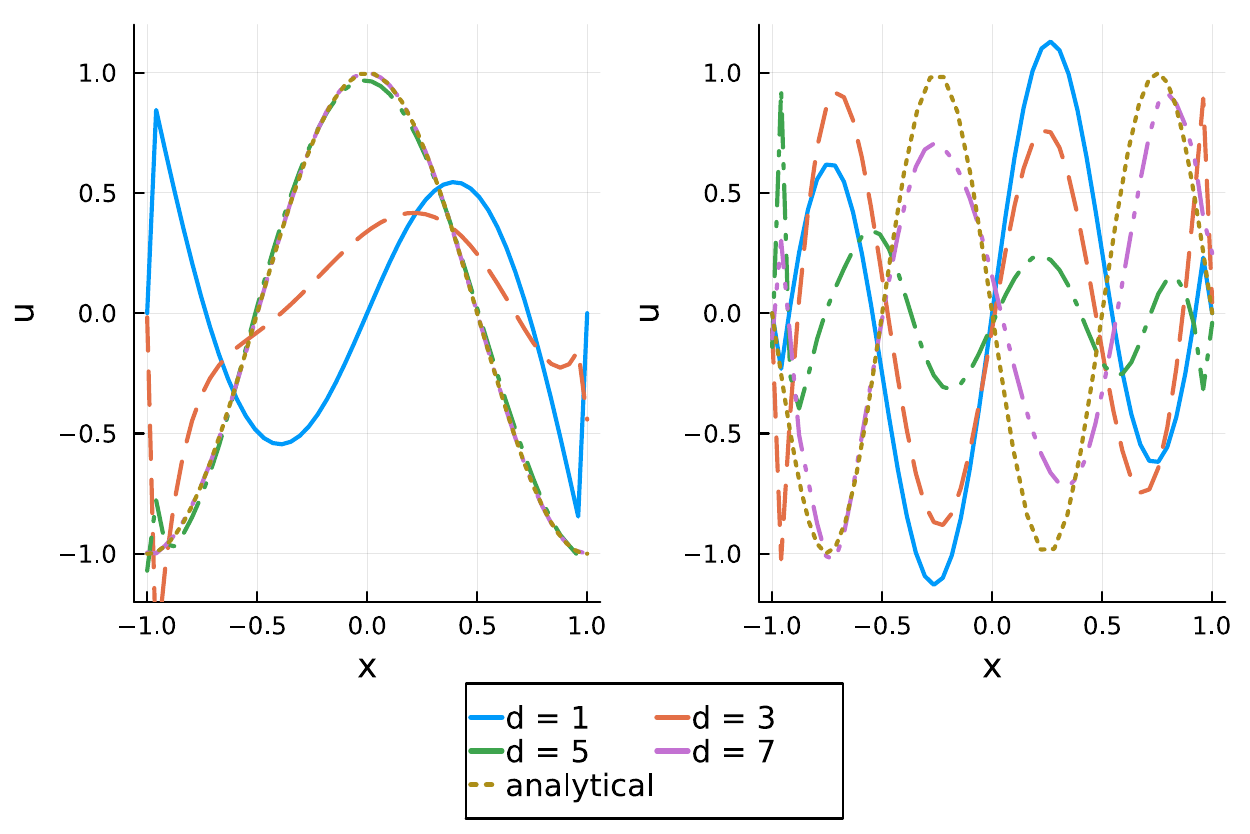}
	\caption{ Numerical solutions of the linear advection equation using FSBP operators with different polynomial degrees $d$ on $N=50$
	nodes and at $t=1.75$. The left subplot shows the solution for $k=1$ and the right subplot for $k=2$ in the initial condition $u_0(x) =
	\sin(k\pi x)$.
	}
	\label{fig:advection_solutions}
\end{figure}

\begin{table}
\small
\setlength{\tabcolsep}{3pt}
\caption{$L^2$- and $L^\infty$-errors for FSBP and FD-SBP operators with different degrees/orders for the linear advection equation with
$N=50$ nodes, $k = 1$, and $t=1.75$.}
\centering
\begin{tabular}{rccccccccc}
  \toprule
  \textbf{Error} & \multicolumn{6}{@{}c@{}}{\textbf{FSBP}} & \multicolumn{3}{@{}c@{}}{\textbf{Classical FD}} \\
  \cmidrule{2-7}\cmidrule{8-10}
   & \textit{$d = 1$} & \textit{$d = 3$} & \textit{$d = 5$} & \textit{$d = 7$} & \textit{$d = 9$} & \textit{$d = 11$} & \textit{$p = 2$} & \textit{$p = 4$} & \textit{$p = 6$} \\
  \midrule
  \textbf{$L^2$} & 8.3e-01 & 4.7e-01 & 3.9e-02 & 1.9e-03 & 6.2e-05 & \textbf{1.6e-06} & 2.1e-02 & 3.2e-04 & 1.4e-04 \\
  \textbf{$L^\infty$} & 1.8e+00 & 9.8e-01 & 2.1e-01 & 6.8e-03 & 2.4e-04 & \textbf{7.2e-06} & 3.1e-02 & 5.7e-04 & 3.3e-04 \\
  \bottomrule
\end{tabular}
\label{table:advection_errors}
\end{table}

The discrete $L^2$- and $L^\infty$-errors for the different operators are summarized in \Cref{table:advection_errors} for $k = 1$. For
comparison, we also include the errors obtained with classical FD-SBP operators of accuracy orders two, four, and six with a diagonal norm
and standard boundary closures \cite{mattsson2004summation}. Notably, the SBP operators constructed by the optimization procedure yield more
accurate results for larger bases including polynomials up to degree $d=9$ or $d=11$. The highest degree for which the optimization
converged to a minimum with zero residual in this setup was $d=12$. It is worth mentioning that the optimization problem is usually not
significantly harder to solve for larger bases, and the computation time depends more on the number of nodes $N$ than on the basis size.
Typically, larger bases need somewhat more iterations to converge. We observed similar results for other function spaces, e.g., exponential,
trigonometric, and radial basis functions. We omit these results here for brevity.

\section{Enforcing additional structure on FSBP operators}
\label{sec:optimization}

As demonstrated in the previous section, SBP operators that satisfy the standard conditions from \Cref{def:FSBP} are not necessarily well
suited for practical applications, e.g., the numerical solution of hyperbolic PDEs. This raises the question: What additional properties
should an SBP operator have to ensure accurate approximation of derivatives and integrals? The previous section illustrated that unresolved
modes can lead to inaccurate solutions. Unresolved modes are associated with functions that are not differentiated accurately by the SBP
operator. A subclass of unresolved modes is formed by the nullmodes, i.e., non-constant functions that are mapped to zero by the
differentiation matrix. This suggests that the nullspace of the differentiation matrix,
\begin{equation}
	\nul(D) = \{\boldsymbol{u} \in \R^N : D\boldsymbol{u} = 0\},
\end{equation}
plays a crucial role. Since the differentiation matrix should approximate the first derivative operator, it is natural to require that its
nullspace contains only constant functions, i.e., $\nul(D) = \spn\{1\}$. Such SBP operators are called \emph{nullspace consistent}
\cite{svard2019convergence,ranocha2019notes,linders2020properties,linders2022eigenvalue,williams2024full,glaubitz2024summation,barsukow2026stability}.
The importance of nullspace consistency has been highlighted in \cite{svard2019convergence}, where it is one of the assumptions under
which convergence rates of energy-stable finite difference schemes are proven. Since we always include the constant function in the
function space $\F$ to ensure conservation \cite{glaubitz2023summation}, the SBP operator is \emph{consistent}, and problems can arise only
from non-constant elements in the nullspace. A useful and easy-to-verify characterization of nullspace consistency is given by the following
lemma.
\begin{lemma}
	A consistent SBP operator is nullspace consistent if and only if $\rank(D) = N - 1$.
\end{lemma}
\begin{proof}
	Since the SBP operator is consistent, we have $1 \in \nul(D)$ and thus $\dim(\nul(D)) \ge 1$. By the rank-nullity theorem, we have
	$\dim(\nul(D)) + \rank(D) = N$. Therefore, $\nul(D) = \spn\{1\}$ is equivalent to $\dim(\nul(D)) = 1$ and thus to $\rank(D) = N - 1$.
\end{proof}
Another useful characterization of nullspace consistency is given in \cite{linders2022eigenvalue}.
\begin{lemma}\label{lem:linders}
	An SBP operator $D = P^{-1}(S + B/2)$ is nullspace consistent if and only if the matrix
	\begin{equation}\label{eq:D-tilde}
		\tilde D = D + \nu P^{-1}\mathbf{e}_L\mathbf{e}_L^T
	\end{equation}
	is invertible for some $\nu \neq 0$.
\end{lemma}
\begin{proof}
	See \cite[Lemma 1]{linders2022eigenvalue}.
\end{proof}
It is known that nullspace consistency is not automatically guaranteed by the SBP conditions. Explicit counterexamples are presented in
\cite{svard2019convergence,linders2020properties,linders2022eigenvalue}. In many cases, an even stronger property than nullspace consistency
is required, the \emph{eigenvalue property} \cite{linders2022eigenvalue}: all eigenvalues of $\tilde D$ in \eqref{eq:D-tilde} have positive
real part for any $\nu > 1/2$.

The FSBP operators constructed in \Cref{sec:counterexample} were obtained from \eqref{eq:optimization-problem-parametrized} with the
full skew-symmetric part \eqref{eq:skew-symmetric}, that is, without imposing any additional structure. We call such operators
\emph{unstructured dense} operators to distinguish them from dense operators that do carry additional structure, such as the
collocation-based operators on Gauss or Gauss--Lobatto nodes mentioned in \Cref{sec:background}. Coming back to these examples, we observe
that the unstructured dense FSBP operators have ranks significantly smaller than $N - 1$, ranging from $3$ for $d = 1$ to $24$
for $d = 11$ (with $N = 50$). These operators do not satisfy the eigenvalue property either. We denote by $n_-$ the number of
eigenvalues of $\tilde D$ in \eqref{eq:D-tilde} with $\nu = 1$ that have non-positive real part; the eigenvalue property requires $n_- = 0$.
For the unstructured dense operators, $n_-$ equals $N - \rank(D) - 1$, i.e., the dimension of the nullspace minus one, see also
\Cref{lem:stationary}. This indicates that many nullmodes exist, which can explain the poor performance, especially for higher frequencies.

Two observations explain why such operators occur. The first is that the construction problem is severely underdetermined.
\begin{lemma}\label{lem:affine}
	Let the nodes $X$ and a basis $\{f_j\}_{j = 1}^K$ of $\F$ be fixed and let $\Sigma$ be a linear space of admissible skew-symmetric
	matrices, for instance all of them or those with a prescribed sparsity pattern. Then
	\begin{equation}\label{eq:affine-set}
		\mathcal{A} = \{(S, P) \in \Sigma\times\mathcal{D} \mid XW + BV/2 = 0\},
	\end{equation}
	where $\mathcal{D}$ is the space of diagonal matrices, is an affine subspace, and the $\F$-exact FSBP operators $D = P^{-1}(S + B/2)$
	are obtained by intersecting $\mathcal{A}$ with the open cone $\{P \mid p_i > 0\}$. If $\mathcal{A} \neq \emptyset$, then
	\begin{equation}\label{eq:affine-dim}
		\dim\mathcal{A} = n_\mathrm{p} - \rank(\mathcal{L}) \ge n_\mathrm{p} - NK,
	\end{equation}
	where $n_\mathrm{p} = \dim\Sigma + N$ is the number of free parameters and $\mathcal{L}\colon (S, P)\mapsto SV - PV_x$.
\end{lemma}
\begin{proof}
	Conditions (ii) and (iii) of \Cref{def:FSBP} hold by construction for $P$ diagonal with positive entries and $S$ skew-symmetric.
	Condition (i) reads $P^{-1}(S + B/2)V = V_x$, which is equivalent to $SV - PV_x + BV/2 = XW + BV/2 = 0$. Since $\mathcal{L}$ is linear,
	this is an inhomogeneous linear system of $NK$ scalar equations in $n_\mathrm{p}$ unknowns. Its solution set is therefore an affine
	subspace of dimension $n_\mathrm{p} - \rank(\mathcal{L})$ if it is nonempty, and $\rank(\mathcal{L}) \le NK$.
\end{proof}
Hence the construction does not determine a single operator but an entire affine family, and which member of it is returned depends on
the optimization algorithm and its initialization. For the unstructured dense ansatz, this family is large: it has dimension $64$ for $N =
15$ and $\F = \mathcal{P}_3$, and dimension $1127$ for $N = 50$ and $\F = \mathcal{P}_2$, see \Cref{table:solution_set}. The second
observation is that the algorithm does not select an arbitrary member of this family: from its default initialization, the member it selects
is rank-deficient by construction.
\begin{proposition}\label{prop:minnorm}
	Let $\Sigma$ consist of all skew-symmetric matrices, let $\mathcal{A} \neq \emptyset$, and let $(S, P)$ be the element of
	$\mathcal{A}$ with minimal Frobenius norm. Then $\rank(S) \le 2K$ and $\rank(S + B/2) \le 2K + 2$. If, in addition, $1 \in \F$, then
	$\rank(S + B/2) \le 2K$. Consequently, if $P$ is invertible, then
	\begin{equation}\label{eq:rank-bound}
		\rank(D) \le 2K + 2, \quad\text{and}\quad \rank(D) \le 2K \quad\text{if } 1 \in \F.
	\end{equation}
\end{proposition}
\begin{proof}
	The minimum-norm solution of a consistent linear system lies in the range of the adjoint of the associated linear map, see, e.g.,
	\cite[Section 5.5]{golub2013matrix}, so we compute $\mathcal{L}^*$. Equipping $\Sigma\times\mathcal{D}$ and $\R^{N\times K}$ with the
	Frobenius inner product and using $\scp{AB}{C} = \scp{A}{CB^T}$,
	\begin{equation*}
		\scp{\mathcal{L}(S, P)}{\Lambda} = \scp{S}{\Lambda V^T} - \scp{P}{\Lambda V_x^T}
		\quad\text{for all } (S, P)\in\Sigma\times\mathcal{D},\ \Lambda\in\R^{N\times K}.
	\end{equation*}
	The matrices $\Lambda V^T$ and $\Lambda V_x^T$ are neither skew-symmetric nor diagonal in general. Since $S$ and $P$ are, only the
	orthogonal projections onto $\Sigma$ and $\mathcal{D}$ contribute, because symmetric matrices are orthogonal to $\Sigma$ and matrices
	with vanishing diagonal are orthogonal to $\mathcal{D}$. These projections are $\Pi_\Sigma(M) = (M - M^T)/2$ and $\Pi_\mathcal{D}(M) =
	\diag(M_{11}, \ldots, M_{NN})$, so that
	\begin{equation*}
		\mathcal{L}^*(\Lambda) = \left(\tfrac{1}{2}(\Lambda V^T - V\Lambda^T),\, -\Pi_\mathcal{D}(\Lambda V_x^T)\right).
	\end{equation*}
	The minimum-norm solution therefore has $S = (\Lambda V^T - V\Lambda^T)/2$ for some $\Lambda\in\R^{N\times K}$. Since $V$ has $K$
	columns, $\rank(\Lambda V^T) \le K$ and $\rank(V\Lambda^T)\le K$, whence $\rank(S)\le 2K$. Since $\rank(B) = 2$, it follows that
	$\rank(S + B/2) \le \rank(S) + \rank(B) \le 2K + 2$.

	Now let $1 \in \F$. Applied to the constant function, whose derivative vanishes, the exactness conditions $SV - PV_x + BV/2 = 0$ yield
	$(S + B/2)\mathbf{1} = 0$, i.e., $B\mathbf{1} = V\Lambda^T\mathbf{1} - \Lambda V^T\mathbf{1}$. Hence $\mathbf{e}_R - \mathbf{e}_L =
	B\mathbf{1}$ lies in the span of the columns of $V$ and $\Lambda$, so that the space $\mathcal{U}$ spanned by these columns together
	with $\mathbf{e}_L$ and $\mathbf{e}_R$ has dimension at most $2K + 1$. Since $(S + B/2)^T = (V\Lambda^T - \Lambda V^T + B)/2$, the range
	of $(S + B/2)^T$ lies in $\mathcal{U}$. It is moreover orthogonal to the nullspace of $S + B/2$, which contains $\mathbf{1}$, and
	$\mathbf{1}$ is not orthogonal to $\mathcal{U}$, since $\mathbf{e}_L^T\mathbf{1} \neq 0$. Therefore, $\rank(S + B/2) \le
	\dim(\mathcal{U} \cap \{\mathbf{1}\}^\perp) \le \dim(\mathcal{U}) - 1 \le 2K$. Finally, $\rank(D) = \rank(S + B/2)$ if $P$ is
	invertible.
\end{proof}
The bound \eqref{eq:rank-bound} is independent of $N$. Since we always include the constant function in $\F$, the sharper bound
$\rank(D) \le 2K$ applies to all operators considered here. Whenever the function space is small compared to the number of nodes, i.e., $2K
< N - 1$, the minimum-norm element of $\mathcal{A}$ can therefore not define a nullspace-consistent operator. We cannot prove that the
optimization converges to this element, but for the initialization used here, the returned operators satisfy the bound, see
\Cref{sub:sensitivity}. Moreover, the ranks we observe agree with this bound and attain it for $K \ge 4$: the operators of
\Cref{sec:counterexample} have ranks $3$, $8$, $12$, $16$, $20$, and $24$ for $K = 2, 4, 6, 8, 10$, and $12$. This identifies a mechanism
behind the large nullspaces. It is not that a nullspace-consistent operator fails to exist in $\mathcal{A}$, but that the element the
optimization returns has the form of the minimum-norm element, and such an element is rank-deficient for structural reasons.

The relevance of the minimum-norm solution is not confined to the regime considered here. For FSBP operators on generalized Gaussian
quadrature nodes, where $K$ is close to $N$ and \eqref{eq:rank-bound} is vacuous, \cite{bercik2026construction} report that the minimum-norm
operator is nullspace inconsistent as well, with rank $N - 2$, and that it is badly conditioned, so that the error of the resulting scheme
grows under mesh refinement. The failure is milder there than in the regime $K \ll N$ studied here, where the rank drops to at most $2K$,
but it points in the same direction: it is the selection of a minimum-norm element, not the exactness conditions themselves, that produces
the nullmodes.

\begin{remark}
	An approach to avoiding undesirable nullmodes, and thus to ensuring nullspace consistency for FSBP operators, is to enforce the
	eigenvalue property during the construction of FSBP operators by including an additional constraint in the optimization problem
	\eqref{eq:optimization-problem-parametrized}, yielding
	\begin{equation}
		\begin{split}
		\min_{\boldsymbol{\rho}\in\R^N, \boldsymbol{\sigma}\in\R^L}\|X(\boldsymbol{\rho}, \boldsymbol{\sigma})W + BV/2\|^2\quad\text{s.t.}\\
		\Re(\lambda) > 0\quad\text{for all }\lambda \in \spec\left(P(\boldsymbol{\rho})^{-1}(S(\boldsymbol{\sigma}) + B/2) + \nu
		P(\boldsymbol{\rho})^{-1}\mathbf{e}_L\mathbf{e}_L^T\right),
		\end{split}
	\end{equation}
	where $\nu > 1/2$ is a fixed parameter, $\Re(\lambda)$ is the real part of the eigenvalue $\lambda$, and $\spec$ denotes the spectrum of
	a matrix. Since $D = P^{-1}(S + B/2)$, this constraint ensures the eigenvalue property and hence also nullspace consistency. However,
	enforcing a condition on the eigenvalues directly is considerably harder than solving the unconstrained problem. Moreover, the
	eigenvalue property alone does not guarantee accurate operators: the unstructured dense operators obtained from random initial guesses
	in \Cref{sub:sensitivity} satisfy the constraint but resolve the physical mode poorly. We therefore consider two other
	approaches and discuss their impact on the nullspaces and on the accuracy of the resulting FSBP operators.
\end{remark}

All numerical examples in this section use Trixi.jl \cite{ranocha2022adaptive,schlottkelakemper2025trixi} for solving the hyperbolic PDEs
and SummationByPartsOperatorsExtra.jl \cite{lampert2026sbpextra} for constructing and applying the SBP operators using the framework of
SummationByPartsOperators.jl \cite{ranocha2021sbp}. For the time integration, we use error-based adaptive explicit Runge--Kutta methods from
OrdinaryDiffEq.jl \cite{rackauckas2017differentialequations} with a relative and absolute tolerance of $10^{-14}$ for the convergence tests
and $10^{-10}$ for the other examples, so that the error of the time integration is negligible compared with the spatial
error. All examples are reproducible using the code provided in our reproducibility repository \cite{glaubitz2026summationRepro}.

\subsection{Using sparse differentiation matrices}
\label{sub:optimization_unconstrained_block}
By defining the skew-symmetric part $S$ of the differentiation matrix as in \eqref{eq:skew-symmetric}, the resulting differentiation matrix
$D$ will generally be dense. This can be desirable in some cases, e.g., because the accuracy is not reduced at the boundaries. However,
dense differentiation matrices are typically applied locally. Here, we consider how sparsity can be integrated into the optimization-based
construction procedure to obtain FSBP operators that are applicable globally or on larger blocks. Such operators are used, for instance, in
global and multiblock FD methods as well as in spectral schemes. Classical FD-SBP operators are banded by construction, and we are
interested in FSBP operators that can take their place on large blocks or globally. Enforcing a sparsity pattern also reduces the cost of
applying the operator from $\O(N^2)$ to $\O(Nb)$, where $b$ denotes the bandwidth.

Beyond this practical motivation, there is a precise reason why a banded ansatz helps here, and it concerns the size of the function space
rather than density as such. By \Cref{prop:minnorm}, the minimum-norm element of $\mathcal{A}$ satisfies $\rank(D) \le 2K$, so it can be
nullspace consistent only if $2K \ge N - 1$. For collocation operators this is no restriction: a Gauss--Lobatto operator on $N$ nodes is
exact on polynomials of degree $N - 1$, so that $K = N$ and the bound is vacuous. Such operators are dense and highly accurate, and nothing
in what follows speaks against them. The same holds beyond the polynomial case: generalized Gaussian quadratures provide nodes that are
optimal for a prescribed function space and yield FSBP operators of minimal dimension, i.e., operators with $K$ close to $N$
\cite{hale2026summation,bercik2026construction,glaubitz2026gaussian}. Choosing few, well-placed nodes is thus one way of avoiding the regime
in which \eqref{eq:rank-bound} binds. Such operators are applied on elements, and resolution is obtained by using many of them, as in DG
methods, so that $K$ and $N$ remain coupled under refinement. Note, however, that the number of nodes of a generalized Gaussian quadrature
is determined by the dimension of the space spanned by the derivatives of the products $fg$ with $f, g \in \F$, cf.\ \Cref{sec:background},
rather than by $K$. If this space is large compared with $\F$, then $N$ exceeds $K$, the FSBP operator on these nodes is no longer unique,
and the questions addressed in this work arise again. We are concerned with the complementary situation, in which a single large
differentiation matrix is used with a fixed function space, as in finite difference and multiblock methods. This setting is attractive,
e.g., because the nodes can be prescribed by the application and because the spectral radius of classical FD-SBP operators grows only
linearly with $N$, instead of quadratically as for collocation on Gauss--Lobatto nodes \cite[Sect.~4.3.2]{canuto2006spectral}, which permits
larger explicit time steps. In this setting, $N$ grows while $K$ stays fixed, so that $K \ll N$, which is precisely where the bound binds. A
banded ansatz then both removes the mechanism behind \Cref{prop:minnorm} and bounds the nullity independently of $N$, as we show below.
Additionally, it improves the performance of the optimization itself, since the number of degrees of freedom is reduced significantly. The
application of the operator also becomes more efficient, e.g., for matrix-vector products. It is therefore of interest to enforce a special
structure on the differentiation matrix, as is also the case for many classical polynomial-based SBP operators. This structure can be
formulated as a block matrix with sub-blocks in the upper left and lower right corners accounting for the boundaries and a banded structure
in the interior \cite{mattsson2004summation}. We can enforce this structure within the optimization procedure by adding additional
constraints such that, instead of \eqref{eq:skew-symmetric}, $S$ is written as
\begin{equation}\label{eq:skew-symmetric-block}
	S(\boldsymbol{\sigma}) = \begin{pmatrix}
		M_1 & C_1 & 0\\
		-C_1^T & A & C_2\\
		0 & -C_2^T & M_2
	\end{pmatrix}\in\R^{N\times N}.
\end{equation}
With a bandwidth of $b$ in the interior and a boundary matrix size $c\ge b$ (e.g., $c = 2b$), we have
$M_1(\boldsymbol{\sigma})\in\R^{c\times c}$ and $M_2(\boldsymbol{\sigma})\in\R^{c\times c}$ being skew-symmetric matrices (of the form
\eqref{eq:skew-symmetric}),
\begin{equation}
	A(\boldsymbol{\sigma}) = \begin{pmatrix}
		0 & \dots & \ast & & & & \\
		\vdots & \ddots & \vdots & \ddots & & & \\
		\ast & \dots & 0 & \dots & \ast & & \\
		& \ddots & \vdots & \ddots & \ddots & \ddots & & \\
		& & \ast & \dots & 0 & \dots & \ast \\
		& & & \ddots & \vdots & \ddots & \vdots \\
		& & & & \ast & \dots & 0
	\end{pmatrix}\in\R^{(N - 2c)\times(N - 2c)},
\end{equation}
and
\begin{equation}
	C_1(\boldsymbol{\sigma}) = \begin{pmatrix}
		0 & 0\\
		\tilde C_1 & 0
	\end{pmatrix}\in\R^{c\times(N - 2c)},
\end{equation}
where the first block column is of size $b$ and the second of size $N - 2c - b$ as well as
\begin{equation}\label{eq:C1-structure}
	\tilde C_1(\boldsymbol{\sigma}) = \begin{pmatrix}
		\ast & & & \\
		\ast & \ast & & \\
		\vdots & \ddots & \ddots & \\
		\ast & \dots & \dots & \ast
	\end{pmatrix}\in\R^{b\times b}.
\end{equation}
The structure for $C_2$ is the same as for $C_1^T$. Here, every asterisk represents one potentially non-zero element while respecting the
skew-symmetry. In summary, we have two times $c(c - 1)/2$ unknowns from $M_1$ and $M_2$, two times $b(b + 1)/2$ unknowns from $C_1$ and
$C_2$, and $b(N - 2c - b) + b(b - 1)/2$ unknowns from $A$ resulting in a total of $\tilde L = Nb + b(b + 1)/2 + c^2 - 2bc - c$ unknowns in
the optimization problem for the differentiation matrix. That is, we reduce the number of variables from $L = \O(N^2)$ to $\tilde L =
\O(Nb)$.

By \Cref{lem:affine}, this also shrinks the affine family \eqref{eq:affine-set} of $\F$-exact FSBP operators that the construction can
return, and it does so drastically, see \Cref{table:solution_set}. For $N = 15$ and $\F = \mathcal{P}_3$ the dimension drops from $64$ to
$4$, and for $N = 50$ and $\F = \mathcal{P}_2$ from $1127$ to $52$.

Two effects need to be distinguished here. First, the sparse ansatz leaves far less freedom. It does not determine the operator uniquely,
but it is consistent with our observation that the sparse operators are reproducible, whereas the dense ones effectively depend on the
optimizer, see also \Cref{sub:sensitivity}. Second, and independently of the first effect, the band structure controls the nullspace from
both sides. On the one hand, the bound of \Cref{prop:minnorm} does not carry over: restricting $S$ to a sparsity pattern destroys the
low-rank structure of $(\Lambda V^T - V\Lambda^T)/2$, so the minimum-norm element of the admissible set is no longer forced to be
rank-deficient. On the other hand, the band structure bounds the nullspace from above. If $D$ is banded in the interior with bandwidth $b$
and its outermost diagonal does not vanish there, then $D\boldsymbol{u} = 0$ is a linear recurrence: the $i$-th interior equation determines
$u_{i + b}$ from $u_{i - b}, \ldots, u_{i + b - 1}$. The dimension of the nullspace is therefore bounded in terms of $b$ and $c$ alone,
independently of $N$, and the boundary blocks reduce it further. No such bound is available for the unstructured dense ansatz, and indeed
the nullity we observe there grows linearly with $N$: for $\F = \mathcal{P}_2$ it is $N - 6$ for $N = 20$, $30$, and $50$, whereas the
corresponding sparse operators have a nullity of one in all three cases. It is this second effect that makes the smaller set of admissible
operators a useful one and not merely a smaller one.

\begin{table}
	\caption{Dimension of the affine set \eqref{eq:affine-set} of $\F$-exact FSBP operators for the unstructured dense and the sparse
	ansatz with $b = 3$ and $c = 2b$, computed as the number of free parameters minus the rank of $\mathcal{L}$.}
	\centering
	\begin{tabular}{llccc}
		\toprule
		nodes and function space & ansatz & parameters $n_\mathrm{p}$ & $\rank(\mathcal{L})$ & $\dim\mathcal{A}$ \\
		\midrule
		$N = 15$, $\F = \mathcal{P}_3$ & dense  & $120$  & $56$  & $64$ \\
		$N = 15$, $\F = \mathcal{P}_3$ & sparse & $60$   & $56$  & $4$ \\
		$N = 50$, $\F = \mathcal{P}_2$ & dense  & $1275$ & $148$ & $1127$ \\
		$N = 50$, $\F = \mathcal{P}_2$ & sparse & $200$  & $148$ & $52$ \\
		\bottomrule
	\end{tabular}
	\label{table:solution_set}
\end{table} Note that for the construction in
\eqref{eq:skew-symmetric-block}--\eqref{eq:C1-structure} we need $N\ge 2c + b$.

In addition, classical FD-SBP operators with equidistant points in the interior have repeating stencils in $A$. Moreover, we have $C_2 =
\bar{C_1}$ and $M_2 = \bar{M_1}$, where the bar $\bar{\cdot}$ denotes permutation of rows and columns. By this approach, we reduce the
number of unknowns further to $c(c - 1)/2 + b$ variables. Since we are interested in a framework with general (i.e., non-equidistant) nodes
and general bases (i.e., not translation-invariant), we allow the more general structure of non-repeating stencils. In our reproducibility
repository \cite{glaubitz2026summationRepro}, we provide two numerical examples, which corroborate that for non-equidistant nodes with
polynomial bases and for equidistant nodes with RBF bases the optimization problem converges only for non-repeating stencils.

It is instructive to relate the above construction to classical diagonal-norm FD-SBP operators. These are assembled from a repeating
interior stencil together with a boundary closure, that is, from precisely the structure \eqref{eq:skew-symmetric-block}, and they do not
exhibit the deficiencies discussed in \Cref{sec:counterexample}: for all operators of
\cite{mattsson2004summation,mattsson2014optimal,mattsson2018boundary} that we examined (orders two to eight on $N = 50$ nodes) we obtain
$\rank(D) = N - 1$, and all eigenvalues of $\tilde D$ in \eqref{eq:D-tilde} have positive real part. The same holds for the dense
collocation-based operators on Gauss--Lobatto--Legendre nodes. Locality of the interior stencil is of course inherent to the finite
difference paradigm and was not introduced with nullspace consistency in mind; nevertheless, together with the boundary closure it yields
operators for which these properties hold automatically. The optimization-based procedure can thus be regarded as a generalization of the
classical constructions to arbitrary node sets and function spaces, in which the structure that guarantees these properties is no longer
implied by the construction itself and therefore has to be imposed explicitly. This is what the present section does.

\begin{remark}
	The optimization-based construction framework is formulated in a general setting, which allows for general sets of points and general
	bases, but it can also be used to reconstruct classical SBP operators. For instance, we have been able to reconstruct dense
	Gauss--Lobatto--Legendre operators and classical sparse FD-SBP operators by choosing appropriate nodes, bases, and structures for $D$.
	This includes the SBP operators from \cite{mattsson2004summation,mattsson2014optimal,mattsson2018boundary}. Our framework thus extends
	the existing polynomial-based one: it recovers the classical operators for specific choices and, at the same time, generates more
	general SBP operators for other node distributions and bases.
\end{remark}

In the following, we use the sparsity-enforcing approach to construct sparse FSBP operators and use them to solve hyperbolic PDEs. First, we
consider a global operator on $N = 50$ equidistant points and apply it to the example from \Cref{sec:counterexample}. We use the polynomial
basis $\mathcal{P}_2 = \spn\{1, x, x^2\}$ and the basis $\mathcal{T} = \spn\{1, x, \sin(\pi x), \cos(\pi x)\}$ with different bandwidths $b
= 3, 4, 5, 6$ and $c = 2b$ and compare the resulting operators to the unstructured dense operators. Note that for $\mathcal{P}_2$ and
$b = 2$, we would recover the classical fourth-order FD-SBP operator, but the optimization procedure does not yield an FSBP operator on
$\mathcal{T}$ for a bandwidth of 2. We observe empirically that for polynomial bases $\mathcal{P}_d$, the smallest bandwidth for which the
optimization problem converges successfully is $b = d$, which results in an FD-SBP operator whose interior stencil, away from the
boundary closures, is the classical central stencil of order $2d$. As expected, the results improve significantly compared to the
unstructured dense operators as shown in \Cref{table:advection_errors_sparse} and smaller bandwidths generally yield better
accuracy.

The most important advantage of the optimization-based construction is that it yields FSBP operators on arbitrary finite-dimensional
function spaces, which can give better results than polynomial bases
\cite{glaubitz2023summation,glaubitz2024energy,glaubitz2025optimization}. In the above example, using the trigonometric basis $\mathcal{T}$
improves the accuracy compared to the polynomial basis up to degree $2$ for all bandwidths. In the case of $k = 1$, where the analytical
solution is given by $\sin(\pi (x - at))$, the error for the FSBP operator with the trigonometric basis is purely given by the time
integration error (up to solver tolerances) since the solution is exactly represented by the function space.

\begin{table}
\caption{$L^2$- and $L^\infty$-errors for FSBP operators with different bandwidths $b$ for the linear advection equation with $N=50$ nodes,
$k = 2$, and $t=1.75$ using a polynomial function space $\mathcal{P}_2 = \spn\{1, x, x^2\}$ and a trigonometric function space $\mathcal{T}
= \spn\{1, x, \sin(\pi x), \cos(\pi x)\}$.}
\centering
\begin{tabular}{rcccc}
  \toprule
   & \multicolumn{2}{@{}c@{}}{\textbf{$\mathcal{F} = \mathcal{P}_2$}} & \multicolumn{2}{@{}c@{}}{\textbf{$\mathcal{F} = \mathcal{T}$}} \\
  \cmidrule(lr){2-3}\cmidrule(lr){4-5}
   & \textit{$L^2$} & \textit{$L^\infty$} & \textit{$L^2$} & \textit{$L^\infty$} \\
  \midrule
  $b = 3$ & 2.6e-02 & 6.6e-02 & \textbf{1.2e-03} & \textbf{4.0e-03} \\
  $b = 4$ & 7.0e-02 & 1.6e-01 & 1.5e-02 & 3.1e-02 \\
  $b = 5$ & 1.4e-01 & 2.9e-01 & 2.4e-02 & 1.1e-01 \\
  $b = 6$ & 2.5e-01 & 5.1e-01 & 4.5e-02 & 9.7e-02 \\
  dense & 1.3e+00 & 2.1e+00 & 1.2e+00 & 1.8e+00 \\
  \bottomrule
\end{tabular}
\label{table:advection_errors_sparse}
\end{table}

We also compare sparse and unstructured dense operators in a multiblock setting with $n_b = 8$ blocks and $N = 15$ nodes per block
for the initial condition $u_0(x) = \exp(-x^2/0.1)$ using $\mathcal{P}_3$ and $\mathcal{T}$ as basis on equidistant nodes. For the sparse
operators, we choose a bandwidth of $b = 3$ and $c = 2b$. The results in \Cref{fig:errors_1d} show that the sparse operators outperform the
unstructured dense ones for both bases.

\begin{figure}[tb]
	\centering
	\includegraphics[width=\textwidth]{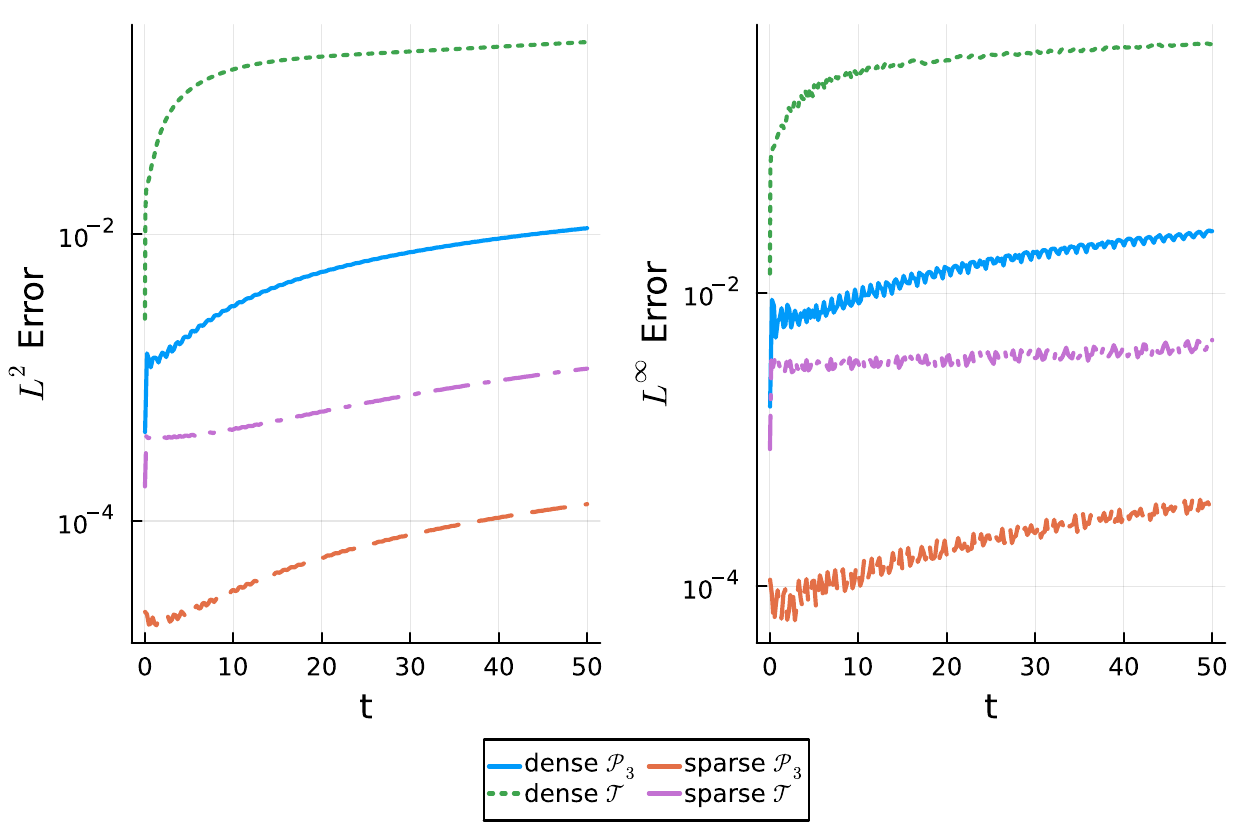}
	\caption{ $L^2$- and $L^\infty$-errors for the one-dimensional linear advection equation using FSBP operators based on $\mathcal{P}_3 =
	\spn\{1, x, x^2, x^3\}$ and $\mathcal{T} = \spn\{1, x, \sin(\pi x), \cos(\pi x)\}$ on $8$ blocks with $N=15$ nodes per block. Sparse
	operators with bandwidth $b = 3$ are compared to unstructured dense operators.
	}
	\label{fig:errors_1d}
\end{figure}

Finally, we use the same sparse and unstructured dense FSBP operators in a tensor product setting on the domain $[-1, 1]^2$ to solve
the two-dimensional compressible Euler equations
\begin{equation}\label{eq:Euler}
	\frac{\partial}{\partial t}
	\begin{pmatrix}
		\rho \\ \rho v_1 \\ \rho v_2 \\ \rho e
	\end{pmatrix}
	+
	\frac{\partial}{\partial x}
	\begin{pmatrix}
		\rho v_1 \\ \rho v_1^2 + p \\ \rho v_1v_2 \\ (\rho e + p) v_1
	\end{pmatrix}
	+
	\frac{\partial}{\partial y}
	\begin{pmatrix}
		\rho v_2 \\ \rho v_1v_2 \\ \rho v_2^2 + p \\ (\rho e + p) v_2
	\end{pmatrix}
	=
	\begin{pmatrix}
		s_1 \\ s_2 \\ s_3 \\ s_4
	\end{pmatrix},
\end{equation}
where $p = (\gamma - 1) ( \rho e - \rho (v_1^2 + v_2^2) / 2 )$ denotes the pressure, $\rho$ the density, $v_1$ and $v_2$ the velocities in
the $x$ and $y$ directions, $e$ the specific total energy, and $\gamma$ the ratio of specific heats. We choose $\gamma = 1.4$, periodic
boundary conditions, and the manufactured solution
\begin{equation}\label{eq:initial_condition_Euler}
\begin{aligned}
	\rho(x, y, t) = c + A\sin(\omega (x + y - t)), \quad
	(\rho v_1)(x, y, t) = \rho(x, y, t), \\
	(\rho v_2)(x, y, t) = \rho(x, y, t), \quad
	(\rho e)(x, y, t) = \rho(x, y, t)^2.
\end{aligned}
\end{equation}
Here, the initial condition is given by the above expression at $t = 0$ and a source term $\mathbf{s} = (s_1, s_2, s_3, s_4)^T$ is added to
the equations to ensure that the manufactured solution is an exact solution of the resulting PDE. For the numerical flux, we choose the HLLC
flux \cite{toro1994restoration}. The $L^2$- and $L^\infty$-errors summed across all variables are shown for $c = 2$, $A = 0.1$, $\omega =
\pi$ in \Cref{fig:errors_2d}, where we use global operators with $N=50$ equidistant nodes. Again, the sparse operators outperform the
unstructured dense ones for both bases. The unstructured dense operator for $\mathcal{P}_3$ crashes after around $t = 4$.

\begin{figure}[tb]
	\centering
	\includegraphics[width=\textwidth]{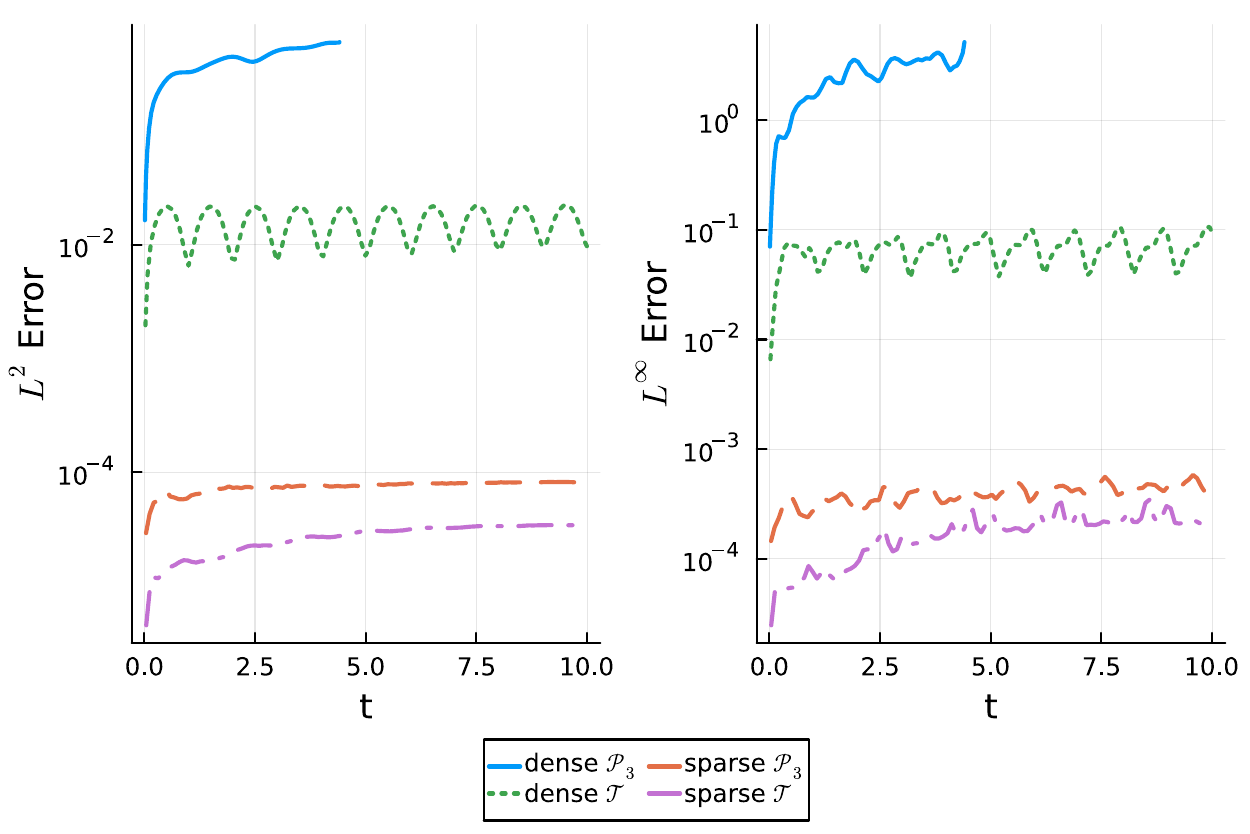}
	\caption{ $L^2$- and $L^\infty$-errors summed across all variables for the two-dimensional compressible Euler equations using FSBP
	operators based on $\mathcal{P}_3 = \spn\{1, x, x^2, x^3\}$ and $\mathcal{T} = \spn\{1, x, \sin(\pi x), \cos(\pi x)\}$ with global
	operators on $N=50$ equidistant nodes. Sparse operators with bandwidth $b = 3$ are compared to unstructured dense operators.
	}
	\label{fig:errors_2d}
\end{figure}

In \Cref{table:eocs_compressible_euler} we show $L^2$-errors and experimental orders of convergence (EOCs) in the density for the
compressible Euler equations using sparse FSBP operators based on $\mathcal{P}_3$ and $\mathcal{T}$. Since the operators are
constructed once on the reference block and the number of nodes $N$ per block is kept fixed under refinement, we expect an $L^2$-convergence
rate of at least $d + 1$, in analogy to element-based methods, where $d$ denotes the largest polynomial degree contained in $\F$. This
yields $4$ for $\mathcal{P}_3$ and at least $2$ for $\mathcal{T}$, where the trigonometric functions could in principle increase the rate
further. For classical FD-SBP operators, convergence rates are known for refinement in the number of nodes: for energy-stable and
nullspace-consistent schemes, a boundary closure of order $p$ combined with an interior order of $2p$ yields the rate $p + 1$
\cite{svard2019convergence}, which carries over to domains split into several blocks coupled across interfaces \cite{svard2021convergence}.
Here, in contrast, $N$ is fixed and the number of blocks grows; see \cite{ranocha2026convergence} for a recent convergence analysis of SBP
discretizations covering this refinement strategy. Note that nullspace consistency enters these results as an assumption. The observed EOCs
agree with these expectations: for $\mathcal{P}_3$ they are still preasymptotic on the coarser grids and settle at about $4$ only for $n_b
\ge 32$ (mean $3.96$), and for $\mathcal{T}$ they settle at about $2$ (mean $2.20$).

\begin{table}[htb]
	\small
	\setlength{\tabcolsep}{4pt}
	\caption{$L^2$-errors and EOCs in the density for the compressible Euler equations and function spaces
	$\mathcal{F}=\mathcal{P}_3$ and $\mathcal{F}=\mathcal{T}$ using sparse FSBP operators with bandwidth $b = 3$ on $n_b$ blocks with
	$N = 15$ nodes each.}
	\begin{subtable}{.5\linewidth}
		\subcaption{$\mathcal{F}=\mathcal{P}_3$}
		\centering
		\begin{tabular}{cccccc}
			\toprule
			$n_b$ & $\rho$ & $\rho v_1$ & $\rho v_2$ & $\rho e$ & EOC \\
			\midrule
			2 & 1.82e-04 & 1.67e-04 & 1.67e-04 & 4.61e-04 & --- \\
			4 & 8.75e-06 & 9.75e-06 & 9.75e-06 & 2.54e-05 & 4.38 \\
			8 & 6.22e-07 & 6.38e-07 & 6.38e-07 & 1.75e-06 & 3.81 \\
			16 & 6.19e-08 & 5.63e-08 & 5.63e-08 & 1.38e-07 & 3.33 \\
			32 & 4.22e-09 & 3.89e-09 & 3.89e-09 & 1.04e-08 & 3.87 \\
			64 & 1.99e-10 & 2.16e-10 & 2.16e-10 & 7.65e-10 & 4.41 \\
			\bottomrule
		\end{tabular}
	\end{subtable}%
	\begin{subtable}{.5\linewidth}
		\subcaption{$\mathcal{F}=\mathcal{T}$}
		\centering
		\begin{tabular}{cccccc}
			\toprule
			$n_b$ & $\rho$ & $\rho v_1$ & $\rho v_2$ & $\rho e$ & EOC \\
			\midrule
			2 & 7.31e-04 & 6.99e-04 & 6.99e-04 & 1.46e-03 & --- \\
			4 & 1.31e-04 & 1.30e-04 & 1.30e-04 & 4.78e-04 & 2.48 \\
			8 & 3.05e-05 & 3.08e-05 & 3.08e-05 & 1.16e-04 & 2.11 \\
			16 & 6.41e-06 & 6.41e-06 & 6.41e-06 & 2.50e-05 & 2.25 \\
			32 & 1.50e-06 & 1.49e-06 & 1.49e-06 & 5.79e-06 & 2.09 \\
			64 & 3.55e-07 & 3.53e-07 & 3.53e-07 & 1.41e-06 & 2.08 \\
			\bottomrule
		\end{tabular}
	\end{subtable}%
	\label{table:eocs_compressible_euler}
\end{table}

\begin{remark}[Computational cost]
	Enforcing sparsity reduces the computational cost in three ways. First, the number of unknowns in the optimization problem is
	reduced from $L = \O(N^2)$ to $\tilde L = \O(Nb)$, as discussed above. Second, the differentiation matrix requires $\O(Nb)$ instead of
	$\O(N^2)$ storage. For the global operators on $N = 50$ nodes used in this section, $D$ has $302$ nonzero entries for $b = 3$, compared
	to $2452$ for the unstructured dense operator; for $N = 100$, the numbers are $602$ and $9902$. Third, applying the operator requires
	$\O(Nb)$ instead of $\O(N^2)$ operations. Exploiting the sparsity in the implementation, the matrix-vector product becomes faster by a
	factor of about $9$ for $N = 50$ and about $19$ for $N = 100$, consistent with the reduction in the number of nonzero entries. For the
	two-dimensional example of \Cref{fig:errors_2d} this reduces the total runtime by a factor of about two, while the numerical solution
	remains unchanged. The gain is smaller for the multiblock operators with $N = 15$ nodes per block, where the boundary blocks of size $c
	= 2b$ make $41\%$ of the entries nonzero.
\end{remark}

We remark that the unstructured dense operators for $\mathcal{P}_3$ and $\mathcal{T}$ have a rank of $8$ each, and are thus not
nullspace consistent. At the same time, the sparse operators on the blocks with $N = 15$ nodes are numerically nullspace consistent
with a rank of $N - 1$. We stress that this is a statement about these particular operators and not about dense operators in general:
as noted above, the dense collocation-based operators are nullspace consistent as well.

\Cref{table:diagnostics} collects the corresponding diagnostics. We point out three observations. First, the norm matrix is well
conditioned, with $\cond(P) \le 5.3$ for all FSBP operators considered here, so the behavior we observe is not caused by ill-conditioning of
$P$. Second, the ratio $\sigma_{N - 1}/\sigma_1$ separates the operators sharply. It is of the order of $10^{-16}$ or smaller for the
unstructured dense operators, between $10^{-3}$ and $10^{-1}$ for the classical FD-SBP operators and for the sparse operators with $b = 3$,
and it decreases rapidly as the bandwidth grows, down to $3.6\cdot 10^{-9}$ for $b = 6$. The same ordering is visible in $\cond(\tilde D)$,
and it matches the accuracy observed in \Cref{table:advection_errors_sparse}, where smaller bandwidths generally yield smaller errors. An
operator can therefore have rank $N - 1$ and still be close to nullspace inconsistency. The eigenvalue property, measured by $n_-$,
separates the two groups as well: $n_- = 0$ for all sparse and classical operators considered here and $n_- > 0$ for all unstructured dense
ones. Here, we count an eigenvalue as having non-positive real part if $\Re(\lambda) \le 10^{-14}$; for the sparse and classical operators
this choice is uncritical, since $\min_\lambda|\Re(\lambda)| \ge 7.4\cdot 10^{-9}$, i.e., five orders of magnitude above the threshold. The
minimum is attained for the sparse operator with $\F = \mathcal{P}_2$ and $b = 6$, consistent with the observation that the operators
approach singularity as the bandwidth grows. Third, for the global operators with $\F = \mathcal{T}$ on $N = 50$ nodes, the optimization
reaches residuals between $1.5\cdot 10^{-15}$ and $5.0\cdot 10^{-10}$, depending on the bandwidth, except for $b = 4$, where it stagnates at
$1.2\cdot 10^{-9}$, slightly above the threshold of \Cref{rem:conventions}. Only for $b = 6$ is the residual small enough for the numerical
rank to be $N - 1$; for $b = 3, 4, 5$ it is $N$ instead, since the constant is annihilated only up to the residual. With a threshold matched
to the residual, these operators are nullspace consistent as well, but the rank criterion is clearly not sharp for them.

\begin{table}[htb]
	\small
	\setlength{\tabcolsep}{4pt}
	\caption{Diagnostics for representative operators: residual $\|XW + BV/2\|$ of \eqref{eq:optimization_problem}, numerical rank,
	the ratio $\sigma_{N-1}/\sigma_1$, the condition numbers of $P$ and $\tilde D$, and the number $n_-$ of eigenvalues of $\tilde D$ with
	non-positive real part. The eigenvalue property holds if and only if $n_- = 0$.}
	\centering
	\begin{tabular}{llcccccc}
		\toprule
		operator & $N$ & residual & $\rank(D)$ & $\sigma_{N-1}/\sigma_1$ & $\cond(P)$ & $\cond(\tilde D)$ & $n_-$ \\
		\midrule
		dense $\mathcal{P}_3$          & $50$ & $6.4\cdot 10^{-16}$ & $8$  & $3.8\cdot 10^{-18}$ & $1.3$ & $1.2\cdot 10^{18}$ & $41$ \\
		dense $\mathcal{P}_9$          & $50$ & $6.0\cdot 10^{-14}$ & $20$ & $1.2\cdot 10^{-16}$ & $3.9$ & $4.8\cdot 10^{17}$ & $29$ \\
		FD-SBP, order $4$              & $50$ & $7.1\cdot 10^{-16}$ & $49$ & $2.7\cdot 10^{-2}$  & $3.5$ & $7.4\cdot 10^{1}$ & $0$ \\
		sparse $\mathcal{P}_2$, $b = 3$ & $50$ & $8.0\cdot 10^{-13}$ & $49$ & $3.9\cdot 10^{-3}$  & $2.5$ & $5.8\cdot 10^{2}$ & $0$ \\
		sparse $\mathcal{P}_2$, $b = 6$ & $50$ & $2.6\cdot 10^{-14}$ & $49$ & $3.6\cdot 10^{-9}$  & $1.6$ & $3.0\cdot 10^{8}$ & $0$ \\
		sparse $\mathcal{T}$, $b = 3$   & $50$ & $1.7\cdot 10^{-10}$ & $50$ & $7.8\cdot 10^{-3}$  & $3.9$ & $2.9\cdot 10^{2}$ & $0$ \\
		\midrule
		dense $\mathcal{P}_3$          & $15$ & $7.0\cdot 10^{-16}$ & $8$  & $1.7\cdot 10^{-17}$ & $2.2$ & $5.5\cdot 10^{16}$ & $6$ \\
		dense $\mathcal{T}$            & $15$ & $8.9\cdot 10^{-16}$ & $8$  & $1.1\cdot 10^{-16}$ & $2.6$ & $9.1\cdot 10^{15}$ & $6$ \\
		sparse $\mathcal{P}_3$, $b = 3$ & $15$ & $1.3\cdot 10^{-15}$ & $14$ & $9.1\cdot 10^{-3}$  & $3.9$ & $2.6\cdot 10^{2}$ & $0$ \\
		sparse $\mathcal{T}$, $b = 3$   & $15$ & $1.9\cdot 10^{-15}$ & $14$ & $3.2\cdot 10^{-2}$  & $4.8$ & $7.7\cdot 10^{1}$ & $0$ \\
		regularized $\mathcal{P}_3$     & $15$ & $6.6\cdot 10^{-10}$ & $14$ & $6.6\cdot 10^{-15}$ & $3.4$ & $6.1\cdot 10^{16}$ & $2$ \\
		\bottomrule
	\end{tabular}
	\label{table:diagnostics}
\end{table}

\subsection{Minimizing the error for an augmented basis}
\label{sub:optimization_regularization}

Two strategies to increase the efficiency of FSBP operators are (i) keeping the basis fixed and either optimizing the node locations or
minimizing the number of nodes, and (ii) enlarging the basis while keeping the nodes fixed. Inspired by the observation in
\Cref{sec:counterexample} that FSBP operators can become more accurate when the basis is enlarged, we follow the second idea and promote
accuracy for functions that are not contained in the basis. We consider the case of an FSBP operator on $\F$ for which we want to increase
the accuracy by enlarging the basis by another set of functions $\G = \{g_1, \ldots, g_M\}$. However, there might not exist an FSBP operator
that is exact on $\F \cup \spn(\G)$ on the given set of nodes $X$. Therefore, we cannot ensure exactness on $\G$. Instead, we minimize the
error on $\G$, which corresponds to solving a regularized optimization problem. We can interpret $\G$ as a set of previously unresolved
modes, which are controlled better if we additionally minimize their error. More precisely, instead of solving the optimization problem from
\eqref{eq:optimization_problem}, we solve
\begin{equation}\label{eq:optimization-regularized}
		\min_{X \in \mathcal{X}}\|XW + BV/2\|^2 + \sum_{k = 1}^M \lambda_k\|P(D\boldsymbol{g}_k - \boldsymbol{g}_k')\|^2
\end{equation}
for some augmented basis $\G = \{g_1, \ldots, g_M\}$ and weights $\lambda_k > 0$. The idea to minimize the error of an SBP operator for
higher-order polynomials has also been employed in \cite{mattsson2018boundary} to construct polynomial FD-SBP operators. In the regularized
formulation, we cannot aim for the objective function to become zero, and thus cannot guarantee exactness on $\F$. This would particularly
affect conservation, since the constant function $1$ is usually included in $\F$ and needs to be differentiated exactly for the scheme to be
conservative. To avoid these potential shortcomings of \eqref{eq:optimization-regularized}, we instead formulate the first term in
\eqref{eq:optimization-regularized} as a constraint and minimize the error on $\G$ only:
\begin{equation}\label{eq:optimization-regularized-constrained}
		\min_{X \in \mathcal{X}}\sum_{k = 1}^M \lambda_k\|P(D\boldsymbol{g}_k - \boldsymbol{g}_k')\|^2 \quad \text{s.t.} \quad
		\|XW + BV/2\|^2 = 0.
\end{equation}
Note that, in contrast to \eqref{eq:optimization_problem}, the objective of
\eqref{eq:optimization-regularized}--\eqref{eq:optimization-regularized-constrained} cannot be expected to vanish at the minimum. As
discussed in \Cref{sec:background}, its minimizer therefore depends on the choice of the norm and of the weights $\lambda_k$. As in
\Cref{sec:background}, we use the Frobenius norm, which reduces to the Euclidean norm for the vector-valued terms, since it is smooth and
inexpensive to evaluate and hence computationally convenient for the gradient-based optimization. We weight the error on $\G$ by the norm
matrix $P$, which is the same weighting as in the residual of \eqref{eq:optimization_problem}. Omitting this weighting, i.e., minimizing
$\sum_k \lambda_k\|D\boldsymbol{g}_k - \boldsymbol{g}_k'\|^2$, works equally well in our experiments and leads to operators with comparable
properties; we use the weighted form here. We made no attempt to optimize the weights $\lambda_k$ and simply set $\lambda_k = 1$ in our
computational experiments.

\begin{figure}[tb]
	\centering
	\includegraphics[width=\textwidth]{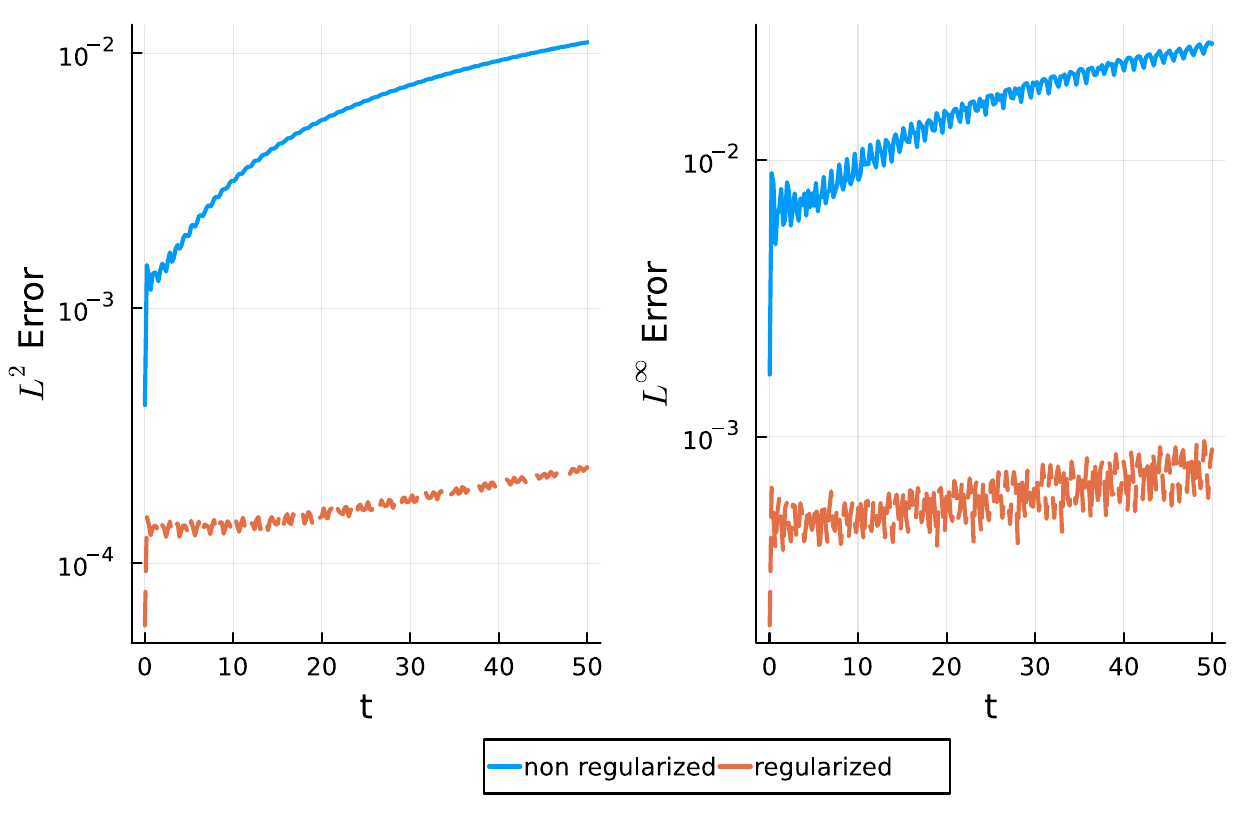}
	\caption{ $L^2$- and $L^\infty$-errors for the one-dimensional linear advection equation using FSBP operators based on $\mathcal{P}_3 =
	\spn\{1, x, x^2, x^3\}$ without regularization and with regularization on $\mathcal{G} = \{\sin(\pi x), \cos(\pi x)\}$ on $8$ blocks
	with $N=15$ nodes per block.
	}
	\label{fig:errors_1d_regularization}
\end{figure}

To demonstrate the potential of this approach, we again consider the one-dimensional linear advection equation with the initial condition
given by $u_0(x) = \exp(-x^2/0.1)$ on $[-1, 1]$ and the two-dimensional compressible Euler equations with manufactured solution
\eqref{eq:initial_condition_Euler} using FSBP operators based on $\mathcal{P}_3 = \spn\{1, x, x^2, x^3\}$ on $8$ blocks with $N = 15$
equidistant nodes per block. We choose $\G = \{\sin(\pi x), \cos(\pi x)\}$ to promote accuracy on these modes and set $\lambda_1 = \lambda_2
= 1$. Note that the optimization without regularization does not converge for $\mathcal{P}_3 \cup \spn(\G)$ to a residual of zero on the
given set of nodes, suggesting that there does not exist an FSBP operator exact on this augmented basis. The $L^2$- and $L^\infty$-errors
are shown in \Cref{fig:errors_1d_regularization} and \Cref{fig:errors_2d_regularization}. We observe that the regularized operators
outperform the non-regularized ones. The numerical rank of the regularized operators is $14 = N - 1$ at the threshold fixed in
\Cref{rem:conventions}, compared to $8$ for the non-regularized ones, i.e., the regularization removes most of the nullmodes. This rank is,
however, not robust: the ratio $\sigma_{N-1}/\sigma_1 = 6.6\cdot 10^{-15}$ lies only slightly above the threshold, and the rank drops to
$12$ once the threshold is relaxed to $10^{-12}$ or larger, so the rank criterion is not meaningful for these operators. What is unambiguous
is that they do not satisfy the eigenvalue property, since $n_- = 2$, and that $\cond(\tilde D) = 6.1\cdot 10^{16}$ is as large as for the
unstructured dense operators. In this respect the regularized operators behave like the unstructured dense operators of
\Cref{sec:counterexample}.

The ranks we observe agree with \Cref{prop:minnorm} once the enlarged space is taken into account. The regularization promotes accuracy on
$\G$ in addition to exactness on $\F$, so the relevant dimension is $K + M = 6$ rather than $K = 4$, and the corresponding bound $2(K + M) =
12$ coincides with the rank obtained for thresholds of $10^{-12}$ and larger. By \Cref{lem:stationary}, it also implies $n_- \ge N - 1 - 12
= 2$, which is the observed value. \Cref{prop:minnorm} does not apply to them directly, since they minimize the objective of
\eqref{eq:optimization-regularized-constrained} over $\mathcal{A}$ rather than the Frobenius norm.

Increasing the number $M$ of regularization functions supports this reading: for $M = 4$ and $M = 6$, where $2(K + M)$ exceeds $N - 1$, the
operators do satisfy the eigenvalue property, with $n_- = 0$ and $\cond(\tilde D) \le 7.7\cdot 10^{2}$ instead of $n_- = 2$ and
$\cond(\tilde D) = 6.1\cdot 10^{16}$ for $M = 2$. The violation of the exactness constraint in
\eqref{eq:optimization-regularized-constrained} grows with $M$, however, and exceeds the threshold of \Cref{rem:conventions}, so that, with
the augmented Lagrangian method used here, enlarging $\G$ trades exactness on $\F$ against the spectral properties. Note also that
the regularized operators are dense, so the improvement is not tied to sparsity; it comes from enlarging the space on which the operator is
accurate.

\begin{figure}[tb]
	\centering
	\includegraphics[width=\textwidth]{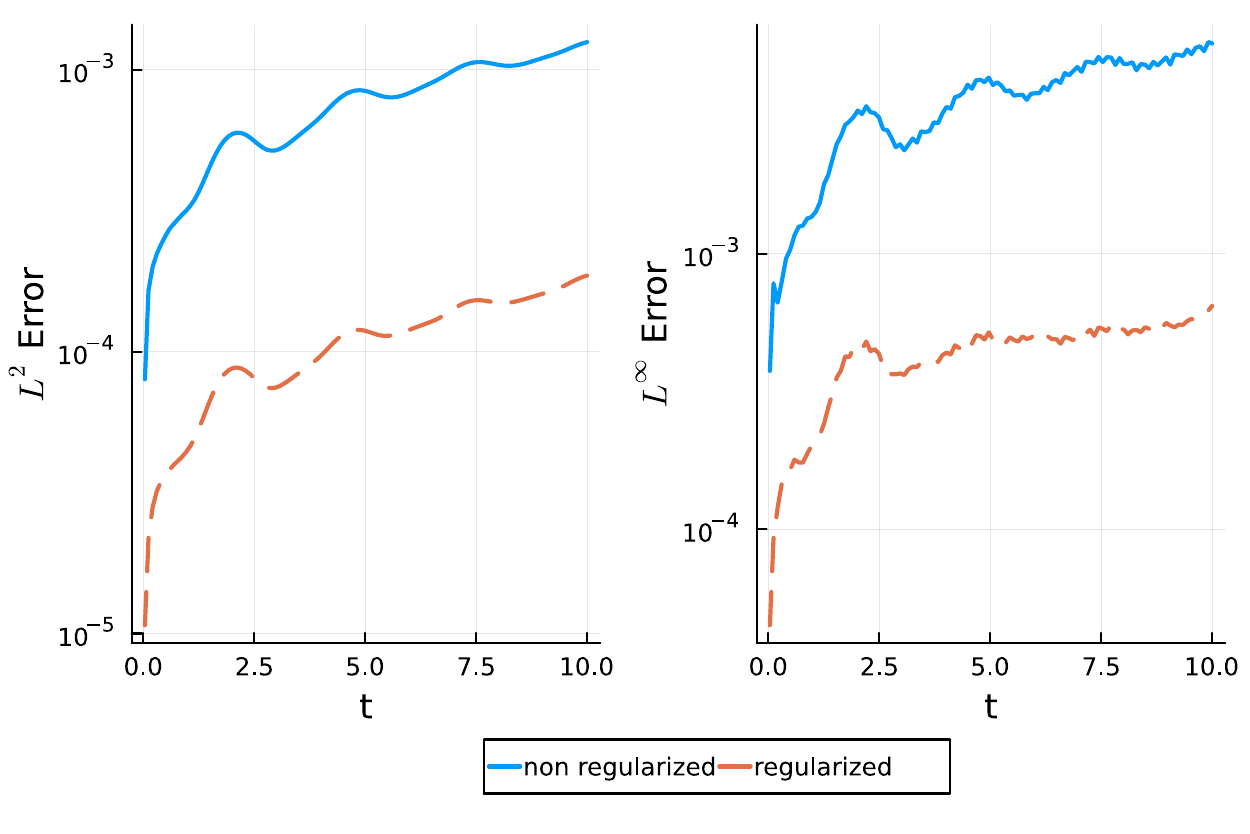}
	\caption{ $L^2$- and $L^\infty$-errors for the two-dimensional compressible Euler equations using FSBP operators based on $\mathcal{P}_3
	= \spn\{1, x, x^2, x^3\}$ without regularization and with regularization on $\mathcal{G} = \{\sin(\pi x), \cos(\pi x)\}$ on $8$ blocks
	with $N=15$ nodes per block.
	}
	\label{fig:errors_2d_regularization}
\end{figure}

\begin{remark}
	To solve the constrained optimization problem, we use the augmented Lagrangian method implemented in the Manopt.jl
	\cite{bergmann2022manopt,bergmann2024manopt} package.
\end{remark}

\subsection{Dispersion and dissipation}\label{sub:dispersion}

The rank, the singular values, and the spectrum of $\tilde D$ characterize the differentiation matrix in isolation. To assess the resulting
schemes, we now analyze how they propagate Fourier modes, following the dispersion and dissipation analysis of \cite{gassner2011comparison}.
This is the first dispersion and dissipation study that involves nonpolynomial FSBP operators---a gap in the
existing literature. For FD-SBP operators, the dispersion error of the interior stencil has also been used as a design criterion in the
construction itself \cite{williams2024full}; the analysis below instead applies to the complete operator on a block, including its boundary
closures. We consider the linear advection equation $u_t + a u_x = 0$ with $a > 0$ on a periodic mesh of identical blocks of length $h$. On
each block, the solution is approximated at $N$ nodes using an SBP operator $D = P^{-1}(S + B/2)$ on the reference block $[-1, 1]$, and
neighboring blocks are coupled by the fully upwind flux, as in the multiblock experiments above. The semidiscretization on block $j$ reads
\begin{equation}\label{eq:multiblock-semidiscretization}
	\frac{\mathrm{d}\boldsymbol{u}_j}{\mathrm{d}t} = -\frac{2a}{h}\Big(D\boldsymbol{u}_j
	+ P^{-1}\mathbf{e}_L\big(\mathbf{e}_L^T\boldsymbol{u}_j - \mathbf{e}_R^T\boldsymbol{u}_{j - 1}\big)\Big).
\end{equation}
Since all blocks are identical, we seek solutions of \eqref{eq:multiblock-semidiscretization} of the form $\boldsymbol{u}_j =
\hat{\boldsymbol{u}}\,e^{\mathrm{i}(j\theta - \omega t)}$, i.e., modes whose phase changes by the scaled wavenumber $\theta = \kappa h$ from
one block to the next, while the vector $\hat{\boldsymbol{u}} \in \C^N$, which describes the mode within a block, is not prescribed. The
exact Fourier mode $e^{\mathrm{i}(\kappa x - \omega t)}$ is of this form with $\hat u_k = e^{\mathrm{i}\theta x_k/2}$, where $x_k \in [-1,
1]$ denote the nodes of the reference block. Inserting this ansatz into \eqref{eq:multiblock-semidiscretization} leads to the eigenvalue
problem
\begin{equation}\label{eq:bloch}
	L(\theta)\hat{\boldsymbol{u}} = \lambda\hat{\boldsymbol{u}}, \quad
	L(\theta) = D + P^{-1}\mathbf{e}_L\big(\mathbf{e}_L - e^{-\mathrm{i}\theta}\mathbf{e}_R\big)^T,
\end{equation}
whose $N$ eigenvalues determine the scaled frequencies $\Omega = \omega h/a = -2\mathrm{i}\lambda$ of the modes of the scheme. For the exact
solution, $\omega = a\kappa$, i.e., $\Omega = \theta$. The deviation of $\Re(\Omega)$ from $\theta$ describes the dispersion error, where
$\Re(\Omega)/\theta$ is the ratio of the numerical and the exact phase velocity, and $\Im(\Omega)$ describes the dissipation, since the
amplitude of a mode evolves as $e^{\Im(\Omega)at/h}$; energy stability implies $\Im(\Omega) \le 0$. Note that the part of $L(\theta)$ that
does not couple neighboring blocks is $\tilde D$ from \eqref{eq:D-tilde} with $\nu = 1$. Since the system matrix of
\eqref{eq:multiblock-semidiscretization} on a periodic mesh of $n_b$ blocks is block circulant, its eigenvalues are precisely those of
$-\frac{2a}{h}L(\theta)$ for $\theta = 2\pi j/n_b$, $j = 0, \ldots, n_b - 1$. We have verified that this system matrix coincides with the
Jacobian of the Trixi.jl semidiscretization used in our experiments up to rounding errors.

The nullspace of $D$ enters this analysis in a simple way.
\begin{lemma}\label{lem:stationary}
	Let $D = P^{-1}Q$ with $Q + Q^T = B$ be an SBP operator with $D\mathbf{1} = 0$, where $\mathbf{1}$ is the vector of ones, and let
	$m = \dim(\nul(D))$. Then the following holds.
	\begin{enumerate}
		\item[(i)]
		$v_1 = v_N$ for every $\boldsymbol{v} \in \nul(D)$, and the subspace $U = \{\boldsymbol{v} \in \nul(D) \mid v_1 = 0\}$ has dimension
		$m - 1$.

		\item[(ii)]
		Every eigenpair $L(\theta)\boldsymbol{v} = \lambda\boldsymbol{v}$ satisfies $\Re(\lambda)\,\boldsymbol{v}^*P\boldsymbol{v} =
		\frac{1}{2}\big|v_1 - e^{-\mathrm{i}\theta}v_N\big|^2$, where $\boldsymbol{v}^*$ denotes the conjugate transpose.

		\item[(iii)]
		Every $\boldsymbol{v} \in U$ satisfies $\tilde D\boldsymbol{v} = 0$ for every $\nu$ and $L(\theta)\boldsymbol{v} = 0$ for every
		$\theta$. For $\theta \notin 2\pi\Z$, the kernel of $L(\theta)$ is spanned by $U$.
	\end{enumerate}
\end{lemma}
\begin{proof}
	(i) Let $\boldsymbol{v} \in \nul(D)$. Since $Q^T = B - Q$ and $Q\mathbf{1} = PD\mathbf{1} = 0$, we have
	\begin{equation*}
		0 = \mathbf{1}^T P D\boldsymbol{v} = \mathbf{1}^T Q\boldsymbol{v} = \big(Q^T\mathbf{1}\big)^T\boldsymbol{v}
		= \big(B\mathbf{1}\big)^T\boldsymbol{v} = v_N - v_1.
	\end{equation*}
	Since $\mathbf{1} \in \nul(D)$ has a nonzero first entry, the linear functional $\boldsymbol{v} \mapsto v_1$ does not vanish on
	$\nul(D)$, and hence $\dim U = m - 1$.

	(ii) Since $Q$ is real, the complex conjugate of $\boldsymbol{v}^*Q\boldsymbol{v}$ is $\boldsymbol{v}^*Q^T\boldsymbol{v}$, and hence
	$\Re(\boldsymbol{v}^*Q\boldsymbol{v}) = \frac{1}{2}\boldsymbol{v}^*(Q + Q^T)\boldsymbol{v} = \frac{1}{2}\boldsymbol{v}^*B\boldsymbol{v}
	= \frac{1}{2}\big(|v_N|^2 - |v_1|^2\big)$. With $PL(\theta)\boldsymbol{v} = Q\boldsymbol{v} + \mathbf{e}_L\big(v_1 -
	e^{-\mathrm{i}\theta}v_N\big)$, taking the real part of $\lambda\,\boldsymbol{v}^*P\boldsymbol{v} =
	\boldsymbol{v}^*PL(\theta)\boldsymbol{v}$ yields
	\begin{equation*}
		\Re(\lambda)\,\boldsymbol{v}^*P\boldsymbol{v} = \frac{1}{2}\big(|v_N|^2 - |v_1|^2\big) + |v_1|^2
		- \Re\big(e^{-\mathrm{i}\theta}\overline{v_1}v_N\big) = \frac{1}{2}\big|v_1 - e^{-\mathrm{i}\theta}v_N\big|^2.
	\end{equation*}

	(iii) For $\boldsymbol{v} \in U$, we have $\mathbf{e}_L^T\boldsymbol{v} = \mathbf{e}_R^T\boldsymbol{v} = 0$ by (i), and thus $\tilde
	D\boldsymbol{v} = D\boldsymbol{v} = 0$ and $L(\theta)\boldsymbol{v} = D\boldsymbol{v} = 0$. Conversely, let $L(\theta)\boldsymbol{v} =
	0$ with $\boldsymbol{v} \in \C^N$. By (ii), $v_1 = e^{-\mathrm{i}\theta}v_N$, so that $D\boldsymbol{v} = L(\theta)\boldsymbol{v} = 0$.
	Since $D$ is real, the real and imaginary parts of $\boldsymbol{v}$ lie in $\nul(D)$, and (i) yields $v_1 = v_N$. Hence $(1 -
	e^{-\mathrm{i}\theta})v_1 = 0$, i.e., $v_1 = 0$ for $\theta \notin 2\pi\Z$, and the real and imaginary parts of $\boldsymbol{v}$ belong
	to $U$.
\end{proof}

\Cref{lem:stationary} makes precise what we called nullmodes at the beginning of this section. By (ii), the damping $-\Im(\Omega) =
2\Re(\lambda)$ of every mode is determined by its jump $v_1 - e^{-\mathrm{i}\theta}v_N$ at the block interfaces, which is precisely what the
upwind flux penalizes. The $m - 1$ linearly independent elements of $U$ vanish at both boundary nodes and are therefore modes of
\eqref{eq:multiblock-semidiscretization} with $\Omega = 0$ for every wavenumber. We call a mode with $\Omega = 0$ \emph{stationary}. The
modes in $U$ are thus neither transported nor damped, so that any component of the numerical solution in $U$ persists unchanged for all
times, and dissipation introduced through the interfaces cannot damp them. Whether they affect the accuracy therefore depends on how
strongly they are excited. For the nullspace-inconsistent modal collocation operators of \cite{chan2018discretely},
\cite{hicken2026collocation} shows that these components remain zero if the initial data and source terms are projected onto the polynomial
space, in which case the scheme is equivalent to a DG method. By (iii), they are moreover the only stationary modes for $\theta \notin
2\pi\Z$: for these wavenumbers, the stationary modes of the scheme are precisely the nullmodes that vanish at the boundary nodes, and every
nullmode differs from one of them only by a constant. These modes are also present for a single periodic block ($\theta = 0$), i.e., for the
global operators of \Cref{sec:counterexample}, so that nullmodes are not damped at all. Moreover, $U$ is precisely the kernel of $\tilde D$
for every $\nu > 1/2$, so that $\tilde D$ has the eigenvalue zero with geometric multiplicity $m - 1$. For $\nu = 1$, this is shown in the
proof of \cite[Lemma~2]{linders2020properties}, which also contains (i) and the first part of (iii), and (ii) is the analogue of
\cite[Lemma~1]{linders2020properties} for $L(\theta)$. In particular, $n_- \ge m - 1$. The operators in \Cref{table:diagnostics} are
consistent with this. For the unstructured dense operators, $n_-$ attains this lower bound, i.e., equals the nullity minus one ($41$, $29$,
$6$, and $6$), so that $\tilde D$ has no further eigenvalues on the imaginary axis. For the regularized operator, $n_- = 2$ corresponds to
the rank $12$ obtained for thresholds between $10^{-12}$ and $10^{-6}$, and the remaining operators have $n_- = 0$.

Beyond the nullmodes, the analysis shows how accurately the remaining modes are resolved. For each $\theta$, we identify the physical mode
as the eigenvector of $L(\theta)$ closest to the exact Fourier mode $e^{\mathrm{i}\theta x/2}$, $x \in [-1, 1]$, in the norm induced by $P$;
in the resolved range, this coincides with the criterion $\Re(\Omega) \approx \theta$ used in \cite{gassner2011comparison}.
\Cref{fig:dispersion} shows its dispersion error and dissipation for the operators with $N = 15$ nodes per block used in
\Cref{fig:errors_1d,fig:errors_1d_regularization}, together with the classical FD-SBP operator of order four and the dense collocation
operator on $N$ Gauss--Lobatto--Legendre (GLL) nodes. As in \cite{gassner2011comparison}, $\theta$ and $\Omega$ are normalized by the number
of nodes per block, and \Cref{table:dispersion} reports the resolution $\mathrm{PPW}_{\min}(\delta) = 2\pi N/\theta_{\min}(\delta)$ in
points per wavelength required for an error $|\Omega - \theta| \le \delta$, where $\theta_{\min}(\delta)$ is the smallest $\theta$ with
$|\Omega - \theta| > \delta$. \Cref{fig:dispersion} shows that the operators differ mainly in the range of wavenumbers in which the physical
mode is resolved. For the unstructured dense operators, both the dispersion error and the damping exceed $10^{-4}$ already for $\theta/N \le
0.25$, and the physical mode can no longer be identified beyond $\theta/N \approx 0.5$. For the sparse and regularized operators, both
remain below this level up to $\theta/N \approx 0.4$--$0.6$, and the physical mode can be identified up to $\theta/N \approx 0.8$--$1.1$,
similar to the FD-SBP operator, whereas for the GLL operator, the errors remain below $10^{-4}$ up to $\theta/N \approx 1.3$ and the
physical mode can be identified up to $\theta/N \approx 1.9$.

The results show that sparsity does not come at the cost of resolution for the operators constructed here. For $\delta = 10^{-2}$, the
sparse operators require $12.6$ and $9.7$ points per wavelength for $\mathcal{P}_3$ and $\mathcal{T}$, which is comparable to the classical
FD-SBP operator of order four ($15.1$), whereas the unstructured dense operators require $26.0$ and $74.0$. The regularized operator
resolves the physical mode about as well as the sparse ones ($12.0$). This does not mean that dense operators are less accurate in general:
the GLL operator requires only $4.5$ points per wavelength. It is exact for polynomials of degree $N - 1$, i.e., its function space has
dimension $K = N$, and \eqref{eq:rank-bound} does not restrict it.

The operators also differ in their spurious modes, which are summarized in \Cref{table:dispersion}. In line with \Cref{lem:stationary}, the
unstructured dense operators have $m - 1 = 6$ stationary modes. In addition, they possess slow modes with $|\Omega| < 1$ that are only
weakly damped, with $-\Im(\Omega)$ as small as $4.9\cdot 10^{-3}$. The regularized operator has two stationary modes and slow modes with a
damping down to $9.5\cdot 10^{-4}$. The sparse operator for $\mathcal{T}$, the FD-SBP operator, and the GLL operator have no slow spurious
modes at all, and the sparse operator for $\mathcal{P}_3$ has two modes with $\Omega \approx -0.35\mathrm{i}$, which are not transported but
strongly damped. Small singular values of $D$, as reported in \Cref{table:diagnostics}, therefore do not by themselves indicate harmful
near-null modes; what matters is whether the corresponding modes of the scheme are damped. The stationary modes and the physical mode affect
the error differently. Since the stationary modes are neither transported nor damped, the part of the error in them is the difference of the
parts of the initial data and of the exact solution in them. It does not grow in time and vanishes whenever $t$ is a multiple of the period,
but it is not reduced by the dissipation of the scheme either. For the exact solution of \eqref{eq:multiblock-semidiscretization} in the
setting of \Cref{fig:errors_1d,fig:errors_1d_regularization}, the maximum of the norm of this part over the last period $t \in [49, 50]$
amounts to $2.5\,\%$ of the maximal $L^2$-error over this period for the unstructured dense operators and to $7.5\,\%$ for the regularized
operator, since the error is dominated by the error of the physical mode accumulated over time. Under refinement, however, this part
converges more slowly. For the maxima over $t \in [0, 1]$ and $n_b = 4, \ldots, 32$ blocks, it converges with an EOC of about $3$ for the
unstructured dense and the regularized operators for $\mathcal{P}_3$ and of about $1$ for the unstructured dense operator for $\mathcal{T}$,
i.e., one order below the rates $4$ and $2$ expected in connection with \Cref{table:eocs_compressible_euler}. Its share in the error
therefore grows from $14\,\%$ to $40\,\%$ for the unstructured dense operator for $\mathcal{P}_3$, from $14\,\%$ to $46\,\%$ for
$\mathcal{T}$, and from $8\,\%$ to $21\,\%$ for the regularized operator. For the global operators of \Cref{sec:counterexample} at $t =
1.75$, it amounts to $97\,\%$ of the error for $d = 1$ and to between $3.5\,\%$ and $29\,\%$ for $d \ge 3$.

\begin{figure}[tb]
	\centering
	\includegraphics[width=\textwidth]{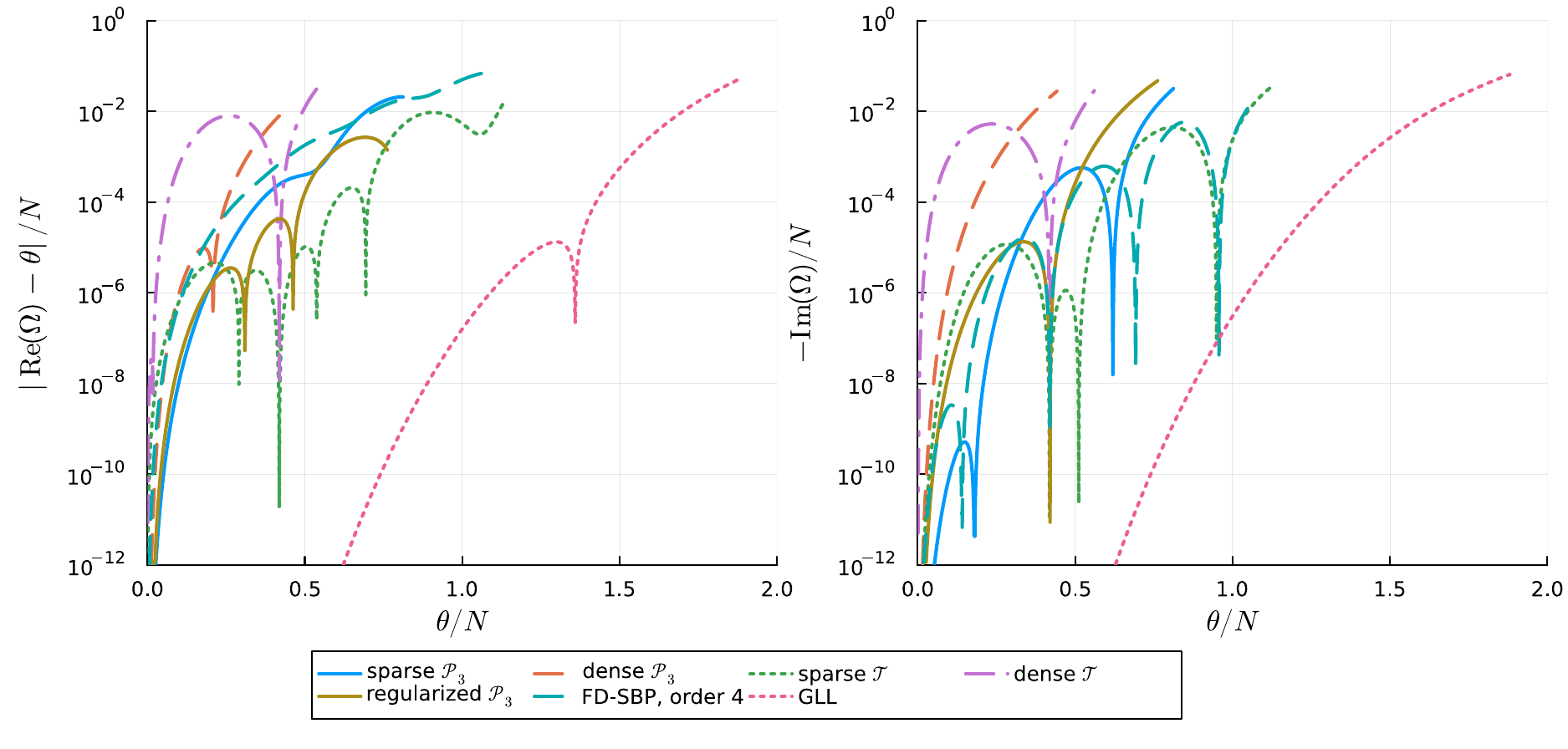}
	\caption{Dispersion error $|\Re(\Omega) - \theta|/N$ (left) and damping $-\Im(\Omega)/N$ (right) of the physical mode for the
	multiblock schemes with $N = 15$ nodes per block. Each curve ends where the physical mode can no longer be identified, i.e., where no
	eigenvector of $L(\theta)$ has an overlap of at least $0.9$ with the exact Fourier mode. The sharp minima correspond to sign changes of
	$\Re(\Omega) - \theta$ (left) and to wavenumbers at which the jump of the physical mode across the block interfaces nearly vanishes,
	cf.\ \Cref{lem:stationary}~(ii) (right); for $\mathcal{T}$, both quantities vanish at $\theta = 2\pi$, where the Fourier mode lies in
	$\mathcal{T}$.}
	\label{fig:dispersion}
\end{figure}

\begin{table}[tb]
	\small
	\caption{Resolution and spurious modes of the multiblock schemes with $N = 15$ nodes per block: points per wavelength
	$\mathrm{PPW}_{\min}(\delta)$ needed for an error $|\Omega - \theta| \le \delta$ of the physical mode, number of stationary modes
	($\Omega = 0$ for all $\theta$), and smallest damping $-\Im(\Omega)$ of the remaining slow modes ($|\Omega| < 1$) for $\theta \in
	[\pi/2, 3\pi/2]$, where the physical mode satisfies $|\Omega| \ge \pi/2$; ``--'' indicates that there are no such modes.}
	\centering
	\begin{tabular}{rccccc}
	  \toprule
	   & \multicolumn{3}{@{}c@{}}{\textbf{points per wavelength}} & \multicolumn{2}{@{}c@{}}{\textbf{spurious modes}} \\
	  \cmidrule(lr){2-4}\cmidrule(lr){5-6}
	   & \textit{$\delta = 10^{-1}$} & \textit{$\delta = 10^{-2}$} & \textit{$\delta = 10^{-3}$} & \textit{stationary} & \textit{weakest damping} \\
	  \midrule
	  sparse $\mathcal{P}_3$ & 9.4 & 12.6 & 18.8 & 0 & 3.5e-01 \\
	  dense $\mathcal{P}_3$ & 18.2 & 26.0 & 36.7 & 6 & 4.9e-03 \\
	  sparse $\mathcal{T}$ & 7.7 & 9.7 & 11.0 & 0 & -- \\
	  dense $\mathcal{T}$ & 33.7 & 74.0 & 137.7 & 6 & 6.7e-03 \\
	  regularized $\mathcal{P}_3$ & 10.1 & 12.0 & 13.4 & 2 & 9.5e-04 \\
	  FD-SBP, order 4 & 9.3 & 15.1 & 24.2 & 0 & -- \\
	  GLL & 4.0 & 4.5 & 5.0 & 0 & -- \\
	  \bottomrule
	\end{tabular}
	\label{table:dispersion}
\end{table}

\subsection{Sensitivity to the initialization}\label{sub:sensitivity}

Since \eqref{eq:optimization-problem-parametrized} is not convex and its global minimizers form the affine family of \Cref{lem:affine}, the
operator returned by the optimization depends on its initialization. All results above use the default initialization $\boldsymbol{\sigma} =
0$ and $s(\rho_i) = 1/N$, i.e., $S = 0$ and $P = \frac{2}{N}I$ on the reference block $[-1, 1]$. To assess how much the conclusions depend
on this choice, we construct the four operators of \Cref{fig:errors_1d} again from $20$ random initial guesses, obtained by adding
independent normally distributed perturbations to the default one, with standard deviation $0.1$ for $\boldsymbol{\sigma}$ and
$\boldsymbol{\rho}$ for ten of them and with standard deviations $1$ for $\boldsymbol{\sigma}$ and $0.5$ for $\boldsymbol{\rho}$ for the
other ten. We chose these standard deviations heuristically, such that the larger perturbations are comparable to the entries of $S$ and to
the variation of the weights $p_i$ in the resulting operators, and the smaller ones are an order of magnitude smaller. All $84$
optimizations reach a residual below $10^{-14}$. \Cref{table:sensitivity} reports the rank of the resulting operators, the resolution
$\mathrm{PPW}_{\min}(10^{-2})$ from \Cref{sub:dispersion}, and the $L^2$-error at $t = 50$ in the setting of \Cref{fig:errors_1d}.

\begin{table}[tb]
	\small
	\caption{Sensitivity of the operators of \Cref{fig:errors_1d} to the initialization of the optimization: numerical rank, points
	per wavelength $\mathrm{PPW}_{\min}(10^{-2})$, and $L^2$-error at $t = 50$ for the default initial guess, as well as the number of
	nullspace-consistent operators, i.e., of operators with rank $N - 1 = 14$, and the ranges (minimum--maximum) of
	$\mathrm{PPW}_{\min}(10^{-2})$ and of the error for $20$ random initial guesses.}
	\centering
	\begin{tabular}{rcccccc}
	  \toprule
	   & \multicolumn{3}{@{}c@{}}{\textbf{default initial guess}} & \multicolumn{3}{@{}c@{}}{\textbf{20 random initial guesses}} \\
	  \cmidrule(lr){2-4}\cmidrule(lr){5-7}
	   & \textit{rank} & \textit{PPW} & \textit{error} & \textit{rank $N - 1$} & \textit{PPW} & \textit{error} \\
	  \midrule
	  sparse $\mathcal{P}_3$ & 14 & 12.6 & 1.3e-04 & 20/20 & 10.0--13.4 & 6.8e-05--2.0e-04 \\
	  dense $\mathcal{P}_3$ & 8 & 26.0 & 1.1e-02 & 20/20 & 23.2--35.4 & 4.5e-03--2.6e-02 \\
	  sparse $\mathcal{T}$ & 14 & 9.7 & 1.2e-03 & 20/20 & 8.4--10.5 & 2.8e-04--2.0e-03 \\
	  dense $\mathcal{T}$ & 8 & 74.0 & 2.2e-01 & 20/20 & 24.5--103.8 & 5.2e-02--3.7e-01 \\
	  \bottomrule
	\end{tabular}
	\label{table:sensitivity}
\end{table}

For the sparse operators, the initialization has little influence. All of them are nullspace consistent and have the eigenvalue property,
they require between $8.4$ and $13.4$ points per wavelength, and their errors differ from those obtained with the default initialization by
at most a factor of five. For the unstructured dense operators, the default initialization yields the rank-deficient operators discussed
after \Cref{prop:minnorm}, whereas every random initialization yields a nullspace-consistent operator with the eigenvalue property. This has
a simple explanation. For the unstructured dense ansatz, the gradient of the objective in \eqref{eq:optimization-problem-parametrized} with
respect to $\boldsymbol{\sigma}$ corresponds to the skew-symmetric matrix $2(RV^T - VR^T)$ with $R = XW + BV/2$, i.e., it has the form
$\Lambda V^T - V\Lambda^T$ from the proof of \Cref{prop:minnorm}. Since the BFGS method starts from a multiple of the identity as
approximation of the inverse Hessian, an induction over its updates shows that every step changes $S$ only by a matrix of this form. For the
default initialization $S = 0$, the returned operator therefore satisfies \eqref{eq:rank-bound}. For a random initial guess, in contrast,
the component of the initial $S$ in the orthogonal complement of these matrices is carried over unchanged. This complement consists of the
skew-symmetric matrices $S$ with $SV = 0$. A more detailed analysis of which elements of $\mathcal{A}$ other optimization methods and
initializations select is beyond the scope of this work. The nullspace-consistent operators obtained in this way nevertheless remain
considerably less accurate than the sparse ones, since they resolve the physical mode poorly, with $23.2$ to $103.8$ points per wavelength.
In this example, every sparse operator yields a smaller error than every unstructured dense operator, so that the advantage of the sparse
operators is not an artifact of a particular initialization. The study moreover separates the two sources of error identified in
\Cref{sub:dispersion}: whether the unstructured dense operators possess stationary modes depends on the initialization, whereas the poor
resolution of their physical mode does not. Nullspace consistency alone therefore does not make these operators accurate.

\section{Summary}
\label{sec:summary}

We have demonstrated that fulfilling the formal definition of an SBP operator alone is not sufficient to guarantee accurate numerical
solutions of hyperbolic conservation laws. By means of illustrative counterexamples, we showed that SBP operators may satisfy all standard
algebraic conditions yet still exhibit poor approximation properties and convergence behavior. Furthermore, we performed a dispersion
and dissipation analysis---the first that includes nonpolynomial FSBP operators. It traced the aforementioned
approximation and convergence deficiencies to two sources, stationary modes and a poorly resolved physical mode. Stationary modes arise from
a lack of nullspace consistency and are neither transported nor damped, and in our experiments their contribution to the error converges
more slowly under mesh refinement. A poorly resolved physical mode dominates the error in long-time simulations and also occurs for
nullspace-consistent operators. The lack of nullspace consistency is not accidental: if the function space is small compared to the number
of nodes, the minimum-norm FSBP operator is rank-deficient (\Cref{prop:minnorm}), and the operators returned by the optimization from its
default initialization have the form of this operator.

To address these issues, we proposed two complementary strategies for improving the quality of SBP operators, formulated as two variants of
the optimization-based construction procedure for FSBP operators presented in \cite{glaubitz2025optimization}. First, we enforced sparsity
in the differentiation matrix, which not only reduces the computational cost and enables the use of FSBP operators as global operators, but
also yields operators with the eigenvalue property, a well-resolved physical mode, and considerably smaller errors, also for random
initializations of the optimization. Second, we discussed regularized optimization formulations that minimize the derivative error for
higher-degree functions beyond the exactness space, thereby improving the resolution of the physical mode, but not removing the
stationary modes.

Nullspace consistency, the eigenvalue property, and the locality of the derivative
stencils are known properties of polynomial SBP operators and are automatically satisfied by the classical constructions. What is new is that they
are not inherited automatically once SBP operators are constructed from the definition alone, and that they can be supplied within the same
optimization-based framework by enforcing sparsity. Together, these results show that nullspace consistency, although a standard
assumption in convergence proofs \cite{svard2019convergence}, is not sufficient for accurate FSBP operators, and that the resolution of the
physical mode has to be taken into account as well. The presented framework is general and applicable to arbitrary node distributions and
function spaces, making it a flexible foundation for developing new high-order discretizations. Future work will focus on extending these
results to meshfree methods for multi-dimensional problems, radial basis function SBP operators, and adaptive node distributions.


\section*{Acknowledgments}
JG acknowledges support by the Swedish Research Council (VR) Starting Grant \#2025-05370, the Zenith Career Development Grant \#26.07, and the National Academic Infrastructure for Supercomputing in Sweden (NAISS) grants \#2025/22-1599 and \#2024/22-1207.
AI and JL acknowledge the support by the Deutsche Forschungsgemeinschaft (DFG)
within the Research Training Group GRK 2583 ``Modeling, Simulation and
Optimization of Fluid Dynamic Applications''.
P\"O is supported by the DFG within SPP
2410, project 525866748 and under the grant 520756621.

\section*{Use of generative AI}
During the revision of this article, Claude Opus 5.5 was used to assist with drafting and editing parts of the text, writing and checking
code and checking numerical results and mathematical arguments. All content was reviewed and verified by the authors, who take full responsibility
for the content of this article.

\bibliographystyle{siamplain}
\bibliography{references}

\end{document}